\documentclass[12pt]{amsart}
\usepackage{amssymb}
\usepackage{amsmath}
\usepackage{bbm}
\usepackage{color}
\usepackage{verbatim}
\usepackage{mathtools}
\usepackage{graphicx} 

\definecolor{darkred}{rgb}{0.75,0,0.3}
\definecolor{DARKRED}{rgb}{0.75,0,0.3}

\newcommand{\adjustedaccent}[1]{%
  \mathchoice{}{}
    {\mbox{\raisebox{-.5ex}[0pt][0pt]{$\scriptstyle#1$}}}
    {\mbox{\raisebox{-.35ex}[0pt][0pt]{$\scriptscriptstyle#1$}}}
                               } 
\newcommand\frownacc[1]{\overset{\adjustedaccent{\frown}}{#1}}
\newcommand{\dol}{\frownacc}

\newcommand\AB{\mathbb A}
\newcommand\AND{\quad\text{and}\quad}
\newcommand\area{\mathsf{area}}
\newcommand{\barra}{\,|\,}
\newcommand\C{\mathbb C}
\newcommand\DD{\mathbb D}
\newcommand\DDbar{\overline\DD}
\newcommand\ep{\varepsilon}
\newcommand\hor{\mathfrak{h}}
\newcommand\im{\mathfrak{i}\,}
\newcommand\Lap{\mathfrak{L}}
\newcommand\leb{\mathsf m}
\newcommand{\Mcal}{\mathcal M}
\newcommand{\Mfrak}{\mathfrak M}
\newcommand\mus{s}
\newcommand\N{\mathbb N}

\DeclareMathOperator{\Poiss}{\boldsymbol \Pi}
\newcommand\R{\mathbb R}

\newcommand\Z{\mathbb Z}

\newtheorem{theorem}{Theorem}
\numberwithin{theorem}{section}
\newtheorem{pro}[theorem]{Proposition}
\newtheorem{lem}[theorem]{Lemma}
\newtheorem{cor}[theorem]{Corollary}

\theoremstyle{definition}
\newtheorem{dfn}[theorem]{Definition}
\newtheorem{rem}[theorem]{Remark}
\newtheorem{disc}[theorem]{Discussion}
\newtheorem{ques}[theorem]{Question}

\begin{document}$\,$ \vspace{-1truecm}
\title{Polyharmonic potential theory on the Poincar\'e disk}
\author{\bf Massimo A. Picardello, Maura Salvatori, Wolfgang Woess}

\address{\parbox{.8\linewidth}{Dipartimento di Matematica\\ 
Universit\`a di Roma ``Tor Vergata''\\
I-00133 Rome, Italy}}
\email{picard@axp.mat.uniroma2.it}

\address{\parbox{.8\linewidth}{Dipartimento di Matematica\\ 
Universit\`a degli Studi di Milano\\
Via Saldini, 50, I-20133 Milano, Italy}}
\email{maura.salvatori@unimi.it}

\address{\parbox{.8\linewidth}{Institut f\"ur Diskrete Mathematik,\\ 
Technische Universit\"at Graz,\\
Steyrergasse 30, A-8010 Graz, Austria}}
\email{woess@tugraz.at}

\thanks{The first author acknowledges support by MIUR Excellence Departments Project awarded to
the Department of Mathematics, University of Rome Tor Vergata,  MatMod@TOV. The second author
acknowledges support as a visiting scientist
at TU Graz. The third author was supported by Austrian Science Fund 
project FWF P31889-N35.}
\subjclass[2020] {31A30; 
                  43A85, 
                  58J32  
                  }
\keywords{Hyperbolic disk, hyperbolic Laplacian, polyharmonic functions, 
          analytic functionals, Riquier problem, Fatou type theorem}
          
          
\begin{abstract}
We consider the open unit disk $\DD$ equipped with the hyperbolic metric and the associated
hyperbolic Laplacian $\Lap\,$. For $\lambda \in \C$ and $n \in \N$, 
a $\lambda$-polyharmonic function of order $n$ is a function $f: \DD \to \C$ such that
$(\Lap - \lambda \, I)^n f = 0$. If $n =1$, one gets $\lambda$-harmonic functions.
Based on a Theorem of Helgason on the latter functions, we prove a boundary integral 
representation theorem for $\lambda$-polyharmonic functions. For this purpose, we 
first determine $n^{\text{th}}$-order $\lambda$-Poisson kernels. Subsequently, we introduce
the $\lambda$-polyspherical functions and determine their 
asymptotics at the boundary $\partial \DD$, i.e., the unit circle.  
In particular, this proves that, for eigenvalues not in the
interior of the $L^2$-spectrum,  the zeroes of these functions do not accumulate at the 
boundary circle. Hence the polyspherical functions can be used to normalise the 
$n^{\text{th}}$-order Poisson kernels.  By this tool, we extend to this setting 
several classical results of potential theory: namely, we study the boundary behaviour of 
$\lambda$-polyharmonic functions, starting with  Dirichlet and Riquier type problems and 
then proceeding to Fatou type admissible boundary limits. 
\end{abstract}

\maketitle
       
\markboth{{\sf M. A. Picardello, M. Salvatori and W. Woess}}
{{\sf Polyharmonic potential theory on the Poncar\'e disk}}
\baselineskip 15pt

\begin{center}
{\it Dedicated to the memory of
Alessandro Fig\`a-Talamanca (1938-2023)}
\end{center}

\medskip

\section{Introduction}\label{sec:intro}
The aim of this article is to initiate a detailed study of the 
potential theory associated with the polyharmonic, and more generally, $\lambda$-polyharmonic 
functions for the hyperbolic Laplacian $\Lap$, that is, the solutions  $f$ of 
$(\Lap - \lambda \, I)^n f = 0$, for $n\in\N$ and $\lambda\in\C$.


A \emph{polyharmonic function of order $n$} on a Euclidean domain $D$ is a
complex-valued function $f$ ond $D$ that belongs to the kernel 
of the $n^{\text{th}}$ iterate of classical Euclidean Laplacian: $\Delta^n f \equiv 0$.
The study of polyharmonic functions goes back to work in the $19^{\text{th}}$ 
century, and continues to be a very active topic. See e.g. the books
by {\sc Aronszajn, Creese and Lipkin}~\cite{ACL} or by 
{\sc Gazzola, Grunau and Sweers}~\cite{GGS}. A classical theorem of 
{\sc Almansi}~\cite{Al} says that if the domain $D$ is star-like with respect
to the origin, then every polyharmonic function of order $n$
has a unique decomposition
$$
f(z) = \sum_{k=0}^{n-1} |z|^{2k}\, h_k(z)\,,
$$
where each $h_k$ is harmonic on $D$, and $|z|$ is the Euclidean length of $z \in D$.
In particular, let the domain be the unit disk
$$
\DD = \{ z= x + \im y \in \C : |z| = \sqrt{x^2 + y^2} < 1 \}.
$$
Assume for the moment that in Almansi's decomposition, each $h_k$ is non-negative.
Then it has an integral representation over the boundary $\partial \DD$ of the disk,
that is, the unit circle, with respect to the \emph{Poisson kernel}
\begin{equation}\label{eq:poisson}
P(z,\xi) = \frac{1 - |z|^2}{|\xi - z|^2} \qquad (z \in \DD\,,\; \xi 
 = e^{\im \phi} \in \partial \DD)
\end{equation} 
Thus, Almansi's decomposition on the disk reads as
\begin{equation}\label{eq:almansi}
f(z) = \sum_{k=0}^{n-1} \int_{\partial \DD} |z|^{2k}\, P(z,\xi)\,d\nu_k(\xi)\,,
\end{equation}
where $\nu_0, \dots, \nu_{n-1}$ are non-negative Borel measures on the unit 
circle. Without requiring non-negativity of the $h_k\,$, the result still
remains true, taking \emph{analytic functionals} (i.e., certain distributions) $\nu_k$ 
instead of  Borel measures: this follows from  results by
{\sc Helgason}~\cite{He}, \cite{Hel} which will be of crucial importance further below.

\smallskip

We now change our viewpoint and view $\DD$ as the \emph{hyperbolic} or 
\emph{Poincar\'e disk} with the hyperbolic length element and resulting metric
\begin{equation}\label{eq:hypmetric}
ds =  \frac{2\sqrt{dx^2 + dy^2}}{1-|z|^2} \AND
\rho(z,w) = \log\frac{|1-z\bar w| + |z-w|}{|1-z\bar w| - |z-w|}.
\end{equation}
We note that the Poisson kernel can be written as 
\begin{equation}\label{eq:busemann}
P(z,\xi) = e^{-\hor(z,\xi)} \quad \text{with}\quad \hor(z,\xi) = 
\lim_{w \to \xi} \Bigl(\rho(w,z) - \rho(w,0)\Bigr)\,,
\end{equation} 
the \emph{Busemann function}. 
The hyperbolic Laplace (or Laplace-Beltrami) operator in the variable $z = x + \im y$ is 
\begin{equation}\label{eq:hypLap}
\Lap = \frac{(1-|z|^2)^2}{4}\, \Bigl( \partial_x^2 + \partial_y^2\Bigr)\,. 
\end{equation}
The harmonic functions for the two Laplacians on the disk clearly 
coincide, but this is no more true for polyharmonic functions of higher order.
While there is an abundant ongoing literature on polyharmonic functions in the
Euclidean setting, we are not aware of an ample body of work for the hyperbolic Laplacian,
or more generally, for Laplace-Beltrami operators on manifolds. A few references
are, for example, {\sc Chung, Sario and Wang}~\cite{ChSaWa}, {\sc Chung}~\cite{Ch} as
well as {\sc Schimming and Belger}~\cite{SchBe} plus some of the citations in the latter paper, 
and also {\sc Jaming}~\cite{Jam}.

The first main aim of this note is to provide an integral representation
in the spirit of \eqref{eq:almansi} for \emph{hyperbolically} polyharmonic functions
of order $n$. More generally, we consider
\emph{$\lambda$-polyharmonic} functions of order $n$, that is, solutions $f: \DD \to \C$ 
of 
$$
(\Lap - \lambda \, I)^n f = 0\,. 
$$
Here, $I$ is the identity operator, and we are taking the $n^{\textrm{th}}$ iterate of  
$\,\Lap - \lambda \, I\,$, where $\lambda \in \C$. If $n=1$, we speak of a 
\emph{$\lambda$-harmonic function.} Considering $\lambda$ 
as an ``eigenvalue'', one should be careful with respect to the space on which
the operator acts. Indeed, we are \emph{not} referring to the action of 
$\Lap$ as a self-adjoint operator on $L^2(\DD,\area_h)$, 
where $\area_h$ is the hyperbolic area measure of $\DD$ and the corresponding spectrum
$(-\infty\,,-\frac14\,]$ is continuous. The mapping
$$
\lambda(\mus) = \mus^2 - \frac14\,,\quad \mus \in \C\,,\; \Re(\mus) \geqslant 0
$$
maps the half-open half plane $\{ \mus \in \C : \Re (\mus) \geqslant 0 \}
\setminus \{ \im t : t < 0\}$
bijectively onto $\C$. 
We write $\mus(\lambda) = \sqrt{\lambda + \frac14}$ for the inverse mapping, where the square root
of $r e^{\im \phi}$ is $\sqrt{r}\, e^{\im \phi/2}$ for $r \geqslant 0$ and $\phi \in (-\pi\,,\,\pi]$. 

Here is our first main result.

\begin{theorem}\label{thm:poly}
Every $\lambda$-polyharmonic function $f: \DD \to \C$ of order $n$ has a unique
representation of the form 
$$
f(z) = \begin{cases}
\displaystyle \sum_{k=0}^{n-1} \int_{\partial \DD} \hor(z,\xi)^k\, 
P(z,\xi)^{\mus(\lambda)+1/2}\,d\nu_k(\xi)\,,\quad&\text{if }\; \lambda \ne -\frac14\,,\\[14pt]
\displaystyle \sum_{k=0}^{n-1} \int_{\partial \DD} \hor(z,\xi)^{2k}\, 
P(z,\xi)^{1/2}\,d\nu_k(\xi)\,,\quad&\text{if }\; \lambda = -\frac14\,,
\end{cases}
$$
where  $\nu_0, \dots, \nu_{n-1}$ are analytic functionals on $\partial \DD$.
\end{theorem}

We postpone the precise definition of these functionals to \S \ref{sec:derive}, and we shall also
rescale the kernels in the integrals in a suitable manner to get the (order $n+1$) 
$\lambda$-polyharmonic Poisson kernels $P_n(z,\xi \,|\, \lambda)$, $n \geqslant 0$; see Proposition \ref{pro:polykernel}. 

Our proof of Theorem \ref{thm:poly} in \S \ref{sec:derive} is inspired 
by  related results obtained in a discrete setting, mostly in recent
work of {\sc Picardello and Woess}~\cite{PW} on polyharmonic functions on general trees, that
was preceded by a long paper by {\sc Cohen et al.}~\cite{CCGS} who had used rather
involved 
methods to obtain an integral representation for polyharmonic functions
on a regular tree with respect to the standard graph Laplacian. (In \cite{PW}, this
is generalised and simplified.) Further motivation for the present work came from 
{\sc Sava-Huss and Woess}~\cite{Sava&Woess}, who studied the boundary behaviour of 
polyharmonic functions on regular trees. There are many profound analogies
between the hyperbolic disk and regular trees. In
the potential theoretic setting considered here, see the first part of 
the note by {\sc Boiko and Woess}~\cite{BW} for an exposition of those 
analogies.\footnote{In the formula for the hyperbolic Laplacian -- 
which is \eqref{eq:hypLap} here -- one of the two
squares is missing~in~\cite{BW}.}  

The natural next goal is to study the asymptoptic behaviour of $\lambda$-polyharmonic functions. For this purpose, but also by inherent interest and for further possible applications, in
\S \ref{sec:poly}, we introduce the family of polyspherical functions $\Phi_n(z \,|\, \lambda)$,
i.e., suitably normalised $\lambda$-polyharmonic functions of order $n+1$ which only depend
on $r = |z|$. Here, the functions $\Phi_0(z \,|\, \lambda)$ for $\lambda \in \C$ are the classical
spherical functions of the Poincar\'e disk. A major step, in itself of interest,
is to determine the asymptotic behaviour 
of $\Phi_n(z \,|\, \lambda)$ near $\partial \DD$,
that is, as $r \to 1$ or equivalently, $R = \rho(z,0) \to \infty\,$; see 
Theorem \ref{thm:asymp}. 

This is important because, for the study of the boundary behaviour of $\lambda$-polyharmonic 
functions one needs  a suitable normalisation of the polyharmonic kernels 
$P_n(z,\xi \,|\, \lambda)$, in order to compensate for their growth or decay; 
this normalisation is then accomplished  in \S \ref{sec:kernels}; it 
extends the classical case $n=0$ via the laborious computations of \S \ref{sec:poly}.
Indeed, when $n=0$, it is well-known that it is appropriate to 
normalise $\lambda$-harmonic functions by the $\lambda$-spherical function, see e.g.   
{\sc Michelson}~\cite{Mich} and {\sc Sj\"ogren}~\cite{Sj}, and for regular trees 
{\sc Kor\'anyi and Picardello}~\cite{KP}. This cannot be done for 
$\lambda \in (-\infty\,,\,-\frac14)$, 
because for these values of $\lambda$ the zeroes of 
the $\lambda$-spherical functions accumulate at the boundary circle, while for all other values of 
$\lambda$, there are no zeroes at all; see Remark \ref{rem:zeroes}.

So for arbitrary $n$, our normalisation consists in dividing by  
$\Phi_n(z \,|\, \lambda)$, which is feasible since it follows from 
Theorem \ref{thm:asymp} that this function has no zeroes close to the boundary circle. 
In \S 4, we show that 
the resulting normalised kernels are good approximate identities at the boundary
points, so that the classical convergence results hold for transforms of functions and
measures on $\partial \DD$; see Proposition \ref{pro:basic_convergence}.

\S  \ref{sec:DirRiq} is dedicated to another important issue: continuous extensions       
from boundary data. We first limit attention to $n=0$ and show that for the normalized kernel 
$P(z,\xi \,|\, \lambda)/\Phi_0(z \,|\, \lambda)$, the solution of the Dirichlet problem with continuous
boundary data is unique for any (even complex) $\lambda \in \C \setminus (-\infty\,,\,-\frac14)$
(Theorem \ref{thm:dirichlet}). Then we 
extend the result to $n>0$ by formulating a suitable version of the Riquier problem, 
adapted to the fact that the quotient of lower and higher order polyspherical functions 
tends to zero at the boundary, and provide such a solution (Corollary \ref{cor:Riq}), 
which is inherently non-unique.

\S \ref{sec:Fatou} answers another fundamental question on the asymptotic 
behaviour of $\lambda$-polyharmonic functions, the Fatou theorem. 
Theorem \ref{thm:Fatou} yields
 admissible non-tangential convergence of the normalised 
transforms of measures on $\partial \DD$ for $\lambda \in \C \setminus (-\infty\,,\,-\frac14\,]$.
For the critical value $\lambda = -\frac14$, we even have
 a wider approach region.  Along the classical guidelines, the proofs
are based on maximal inequalities.

The last \S \ref{sec:exa} is devoted to related examples 
(in the standard case $\lambda =0$), discussions and open questions. 
In particular, we provide all details of an example 
outlined to us by A. Borichev: a harmonic (indeed, analytic) function $h(z)$ 
such that $h(z)/R$ is bounded but has no  radial limits at the boundary, 
as $R \to \infty$ where $R = \rho(z,0) \sim \Phi_1(z|0)$, the biharmonic 
spherical function.
In all the paper, for the reader's benefit, we give most of the details of the (sometimes lengthy) 
computations.

\smallskip

\noindent
\textbf{Acknowledgements.} We acknowlege enlightening 
email conversations with
Alexander Borichev (Marseille), Fausto Di Biase (Pescara), Jean-Pierre Otal (Toulouse)
and Peter Sj\"ogren (G\"oteborg).

\section{Integral representation}\label{sec:derive}

For each $\delta > 0$, consider the space $\mathcal{H}(\AB_{\delta})$ of all
holomorphic functions on the open annulus
$$
\AB_{\delta} = \{ z \in \C : 1-\delta < |z| < 1+\delta \}\,.
$$
The space is equipped with the topology of uniform convergence on
compact sets. The space $\mathcal{H}(\partial \DD)$ of analytic functions on 
the unit circle consists of all functions $g: \partial \DD \to \C$
which possess an extension in $\mathcal{H}(\AB_{\delta})$ for some 
$\delta = \delta(g) > 0$. The topology on  $\mathcal{H}(\partial \DD)$
is the inductive limit of the topologies of $\mathcal{H}(\AB_{\delta})$ 
as $\delta \to 0$.

\begin{dfn}\label{def:functional} An \emph{analytic functional} $\nu$ on
$\partial \DD$ is an element of the dual space of $\mathcal{H}(\partial \DD)$.
We write
$$
\int_{\partial \DD}g\, d\nu := \nu(g)\,,\quad 
g \in \mathcal{H}(\partial \DD)\,.
$$
\end{dfn}
A good way to understand the action of $\nu$ on $\mathcal{H}(\partial \DD)$ is 
described in \cite[p. 114]{Ey}, see {\sc K\"othe}~\cite{Ko}: let 
$\nu_n = \nu(e^{-\im n \phi})$, $n \in \Z$ be the Fourier coefficients of $\nu$.  
Then 
$$
\limsup_{|n| \to \infty} |\nu_n|^{1/|n|} \geqslant 1
$$
(and this characterises the analytic functionals). If $g \in \mathcal{H}(\partial \DD)$
then the Fourier expansion 
$
g(e^{\im \phi}) = \sum_{n \in \Z} g_n \, e^{\im n \phi}
$
is such that 
$$
\limsup_{|n|\to \infty} |g_n|^{1/|n|} < 1. 
$$
Then 
\begin{equation}\label{eq:nuaction}
 \int_{\partial \DD}g\, d\nu = \sum_{n \in \Z} g_n \, \overline {\nu\,}_{\!\!n}\,.
\end{equation}
For more on analytic functionals, resp. hyperfunctions, see e.g. 
{\sc H\"ormander}~\cite[Chapter IX]{Hoe} or {\sc Schlichtkrull}~\cite{Schl}.

It will be useful to write 
$$
P(z,\xi \,|\, \lambda) = P(z,\xi)^{\mus(\lambda)+1/2}.
$$
The following results from an elementary and well-known computation.

\begin{lem}\label{lem:eigen} For $\lambda \in \C$ 
and $\xi \in \partial \DD$, the function $z \mapsto P(z,\xi \,|\, \lambda)$
satisfies
$$
\Lap P(\,\cdot\,,\xi \,|\, \lambda) = \lambda\, P(\,\cdot\,,\xi \,|\, \lambda)\,.
$$  
\end{lem}
We note here that we can write
\begin{equation}\label{eq:extendedP}
P(z,\xi) = \frac{1-|z|^2}{(\xi - z)(1/\xi - \bar z)}\,. 
\end{equation}
In this form, for fixed $z \in \DD$ and any $\lambda \in \C$, 
the function $\xi \mapsto P(z,\xi \,|\, \lambda)$ is in $\mathcal{H}(\AB_{\delta})$ for
$\delta = 1-|z|$. Thus, as a function of $\xi$ in the unit circle, it is in 
$\mathcal{H}(\partial \DD)$. We now recall an important result.

\begin{theorem}[{\sc Helgason} \cite{He}, {\cite[Section V.6]{Hel}}]\label{thm:helgason}
For any $\lambda \in \C$, 
every $\lambda$-harmonic
function $h$ for the hyperbolic Laplacian on the Poincar\'e disk has a
unique representation
$$
h(z) = \int_{\partial \DD} P(z,\xi \,|\, \lambda)\, d\nu(\xi)\,,
$$
where $\nu=\nu^h$ is an analytic functional on $\partial \DD$.
\end{theorem}

A very readable proof of this and several related results
are contained in the beautiful expository paper by {\sc Eymard}~\cite{Ey}.

\begin{rem}\label{rem:nonneg}
If the $\lambda$-harmonic function $h$ is positive real, then $\lambda \geqslant -\frac14$. 
Indeed, it is well known that positive $\lambda$-harmonic functions exist precisely when
$\lambda \geqslant -\frac14$; see e.g. {\sc Sullivan}~\cite[Thm. 2.1]{Sul}. 

For non-negative $h$, the functional $\nu^h$ is a non-negative Borel measure. This 
follows from general Martin boundary theory, see e.g.
{\sc Karpelevi\v{c}}~\cite{Karp} or {\sc Taylor}~\cite{Tay}.
\end{rem}

Before proving Theorem \ref{thm:poly}, we need to find suitable 
polyharmonic versions of the Poisson kernel. That is, for each $n \in \N_0\,$, 
we want to have a kernel of the form 
\begin{equation}\label{eq:ansatz}
P_n(z,\xi \,|\, \lambda) = g_{n,\lambda}\bigl(-\hor(z,\xi)\bigr)\, P(z,\xi \,|\, \lambda)
\end{equation}
which satisfies 
\begin{equation}\label{eq:reduce}
(\Lap - \lambda \, I)^n P_n(z,\xi \,|\, \lambda) = P(z,\xi \,|\, \lambda)\,.
\end{equation}

For this purpose, we shall use the following.

\begin{lem}\label{lem:diff}
Let $f \in C^2(\R)$ and set 
$$
Q_f(z,\xi \,|\, \lambda) = f\bigl(-\hor(z,\xi)\bigr)\, P(z,\xi \,|\, \lambda)\,.
$$
Then 
$$
(\Lap - \lambda \, I)Q_f(z,\xi \,|\, \lambda) = Q_g(z,\xi \,|\, \lambda)\,,
$$
where $g = f'' + 2\mus\, f'$ and $\mus = \mus(\lambda)\,$.
\end{lem}

\begin{proof} Here (and frequently also later) we shall use the fact 
that the Busemann function, hence also the Poisson kernel (as well as the Laplacian), 
are rotation invariant:
\begin{equation}\label{eq:rotation}
\hor(e^{\im  \alpha} z, e^{\im  \alpha} \xi) = \hor(z,\xi) \quad \text{for all }\; z \in \DD\,,\; 
\xi \in \partial \DD.
\end{equation}
Thus, it is sufficient to consider $\xi =1$. Furthermore, from the Poincar\'e disk model
of the hyperbolic plane we can first pass to the upper half plane model via the inverse 
Cayley transform, where in the new coordinates $(u,v)$ with $u \in \R$ and $v > 0$, the 
hyperbolic Laplacian transforms into $v^2 (\partial^2 u + \partial^2 v)$ and the boundary
point $1 \in \partial \DD$ becomes $\im\infty$. Then we make one more change of variables,
setting $w = \log v$ to obtain the \emph{logarithmic model,} where now $(u,w) \in \R^2$ and
the hyperbolic Laplacian becomes
\begin{equation}\label{eq:logmodel}
\Lap = e^{2w} \partial^2_u + \partial^2_w - \partial_w\,.
\end{equation}
In these coordinates, the Busemann function and Poisson kernel at $\im\infty$ are
$$
\hor\bigl((u,w), \im\infty\bigr) = -w \AND P\bigl((u,w), \im\infty\bigr) = e^w\,,
$$
and
$$
Q_f((u,w), \im\infty|\lambda\bigr) = f(w) \, e^{(\mus+1/2)w}\,.
$$
The statement now follows by applying the Laplacian in the form of \eqref{eq:logmodel}.
\end{proof}

\begin{pro}\label{pro:polykernel}
For $\lambda \ne -\frac14$ and $\mus = \mus(\lambda)$, the kernel
$$
P_n(z,\xi \,|\, \lambda) = \frac{1}{n!(2\mus)^n}\,\bigl(-\hor(z,\xi)\bigr)^n\, P(z,\xi)^{\mus +1/2}
$$
satisfies \eqref{eq:reduce}.

For $\lambda = -\frac14$, where $\mus(\lambda)=0$,  identity \eqref{eq:reduce}
holds for
$$
P_n(z,\xi|-\frac14) = \frac{1}{(2n)!}\,\hor(z,\xi)^{2n}\, P(z,\xi)^{1/2}.
$$
\end{pro}

\begin{proof}
In order to find a function $g_{n,\lambda}$ as in \eqref{eq:ansatz}, we start of course with 
$g_{0,\lambda} \equiv 1$. We proceed recursively, 
looking at each step for a function $f_n = f_{n,\lambda}$ such that
\begin{equation}\label{eq:recur}
(\Lap - \lambda \, I) \Bigl[ f_{n,\lambda}\bigl(-\hor(z,\xi)\bigr)\, P(z,\xi \,|\, \lambda)\Bigr] 
= g_{n-1,\lambda}\bigl(-\hor(z,\xi \,|\, \lambda)\bigr)\, P(z,\xi \,|\, \lambda)\,.
\end{equation}
The function $f_n$ will then be replaced by the simpler $g_n = g_{n,\lambda}$ 
which satisfies 
\eqref{eq:ansatz}
before proceeding to $n+1$.
By Lemma \ref{lem:diff}, $f_n$ must solve the differential equation
\begin{equation}\label{eq:equadiff}
f_n'' + 2\mus\, f_n' = g_{n-1}\,,\quad \mus = \mus(\lambda).
\end{equation} 
The characteristic polynomial of \eqref{eq:equadiff} has roots $0$ and $-2\mus$, when $\lambda \ne -\frac14$.
In the latter case, $0$ is a double root.

\smallskip

We start with $\lambda \ne -\frac14$ and $n=1$. Since $g_0=1$, we are 
looking for a special solution of \eqref{eq:equadiff} of the form 
$f_1(w) = A_{1,1}w$, whence $A_{1,1} = 1/( 2\mus)$.  
We get $g_1(w)=f_1(w) = w/( 2\mus)$, and going back to the disc model, 
we obtain $P_1$ via \eqref{eq:ansatz}, as proposed. 
We now prove by induction on $n$ that by setting $g_n(w) = \dfrac{w^{n}}{n!(2\mus)^n}$
we obtain a solution for $P_n\,$. 

Suppose this is true for all orders up to $n-1$. The right hand side of \eqref{eq:equadiff}
is a polynomial of order $n-1$ in $w$. Hence there is  a special solution of the form
$f_n(w) = \sum_{k=1}^n A_{n,k} w^k\,$.  The coefficients $A_{n,k}$ are obtained as
solutions of a system of linear equations, yielding a solution 
of \eqref{eq:recur}.
However, by the induction hypothesis, the terms of order $k < n$ are 
annihilated when applying
$(\Lap - \lambda \, I)^n$ to 
$f_{n,\lambda}\bigl(-\hor(z,\xi)\bigr)\, P(z,\xi \,|\, \lambda)$, so that 
for \eqref{eq:reduce} we only need $g_n(w) = A_{n,n}w^n$.
Inserting $f_n$ into \eqref{eq:equadiff}, comparison of the highest order coefficients 
yields  $\;2\mus n = \dfrac{w^{n-1}}{(n-1)!(2\mus)^{n-1}}$, which completes the
induction step.

\smallskip

When $\lambda = -\frac14$, the differential equation \eqref{eq:equadiff} simplifies to
$f_n'' =  g_{n-1}\,$. Here, we set $g_n = f_n\,$. Starting with $g_0 \equiv 1$, we integrate twice 
at each step  and take just the highest appearing power: $g_1(w) = w^2/2\,$, $g_2(w) = w^4/4!$, and so on, so that $g_n(w) = w^{2n}/(2n)!$, as proposed. 
\end{proof}

Note that when $\lambda = -\frac14$, we even have $(\Lap - \lambda \, I) P_n(z,\xi|-\frac14) 
= P_{n-1}(z,\xi|-\frac14)$. 

In the standard case $\lambda = 0$, we just write
$P_n(z,\xi)$ for $P_n(z,\xi|0)$.

\begin{proof}[{\bf Proof of Theorem \ref{thm:poly}}]
Along with the Poisson kernel, also the function $\xi \mapsto P_n(z,\xi \,|\, \lambda)$
is in $\mathcal{H}(\AB_{\delta})$ for $\delta = 1-|z|$, for every $n \in \N_0$, $\lambda \in \C$
and $z \in \DD$.

We claim that every 
$\lambda$-polyharmonic function $f$ of order $n$ has a unique representation
of the form 
\begin{equation}\label{eq:poly}
f(z) = \sum_{k=0}^{n-1} f_k(z) \quad\text{with}\quad 
f_k(z) = \int_{\partial \DD} P_k(z,\xi \,|\, \lambda)\,d
\nu_k(\xi)\,,
\end{equation}
where  $\nu_0, \dots, \nu_{n-1}$ are analytic functionals on $\partial \DD$.
Furthermore, when $\lambda \geqslant -\frac14$ is real, $\nu_k$ is a non-negative Borel measure 
if and only if 
$$
(\Lap - \lambda \, I)^k g_k \geqslant 0\,,\quad\text{where}\quad 
g_k=f - (f_{n-1} + \dots + f_{k+1})\,.
$$ 
To prove this, we proceed by induction on $n$. For $n=1$, this is Theorem \ref{thm:helgason}. 
Suppose the statement is true for $n-1$.  Let $f$ be $\lambda$-polyharmonic of order $n$.
Then $h = (\Lap - \lambda \, I)^{n-1}f$ is $\lambda$-harmonic. By Theorem
\ref{thm:helgason}, there is a unique analytic functional $\nu_{n-1}$ on $\partial \DD$ such
that
$$
h(z) = \int_{\partial \DD} P(z,\xi \,|\, \lambda)\,d\nu_{n-1}(\xi)\,,
$$
We set 
$$
f_{n-1}(z) = \int_{\partial \DD} P_{n-1}(z,\xi \,|\, \lambda)\,d\nu_{n-1}(\xi)\,.
$$
By Proposition \ref{pro:polykernel}, 
$$
(\Lap - \lambda \, I)^{n-1}f_{n-1} = h = (\Lap - \lambda \, I)^{n-1}f\,.
$$
Thus, $f - f_{n-1}$ is $\lambda$-polyharmonic of order $n-1$, and we can apply the
induction hypothesis to that function in order to get the representation of $f$.
Uniqueness follows from Helgason's Theorem \ref{thm:helgason}. 
The statement on non-negativity for real $\lambda \geqslant -\frac14$ is a 
consequence of
the well known Poisson-Martin representation theorem for positive $\lambda$-harmonic
functions. Since \eqref{eq:poly} differs from the proposed result only 
by the normalisation of the kernels $P_n\,$, the result is proved.
\end{proof}

The previous paper \cite{PW} on trees adopts a different 
method of proof, that could also be used here: it consists in differentiating
$P(z,\xi \,|\, \lambda)$ with respect to $\lambda$ instead of integrating a differential
equation with respect to $z$, which is inherent in the proof of Proposition \ref{pro:polykernel}. The drawback is that this
does not work at the critical value $\lambda = -\frac14$, while on the other hand,
it can be applied to more general domains as long as $\lambda$ belongs to
the $L^2$-resolvent set of the underlying Laplacian.

\emph{From now on, we use the representation formula \eqref{eq:poly} for 
$\lambda$-polyharmonic functions, instead of the one of the statement of
Theorem \ref{thm:poly}.} 

\section{Polyspherical functions}\label{sec:poly}

\begin{dfn}\label{def:polysph}
 For $\lambda \in \C$ and $n \in \N_0\,$, 
 the $n^{\textrm{th}}$
 \emph{$\lambda$-polyspherical funtion} is 
 $$
 \Phi_n(z \,|\, \lambda) = 
 \int_{\partial \DD}  P_n(z,\xi \,|\, \lambda)\,d\xi\,,
 \quad z \in \DD\,,
 $$
where 
$P_n$ is given by Proposition \ref{pro:polykernel}
and 
$d\xi = d\leb(\xi)$ for the normalized Lebesgue measure $\leb$ on the unit circle. 
\end{dfn}

The function $\Phi_n(z \,|\, \lambda)$ is $\lambda$-polyharmonic of order 
$n+1$ and rotation invariant, 
i.e., it depends only on $|z|$.
For $n = 0$, we recover the classical spherical functions 
$\Phi(z \,|\, \lambda)$, where we omit the index $0\,$: for $z=r\,e^{i\phi} \in \DD\,$,
\begin{equation}\label{eq:spherical}
 \begin{aligned}
\Phi(z \,|\, \lambda) 
&=  \int_{\partial \DD}  P(z,\xi \,|\, \lambda)\,d\xi =  \int_{\partial \DD}  
 P(z,\xi)^{\mus(\lambda) +1/2}\,d\xi  \\[3pt] 
 &= \frac1{2\pi} 
 \int_{-\pi}^\pi \Bigl(\frac{1-r^2}{1+r^2 - 2r\cos \phi}\Bigr)^{\mus(\lambda)+1/2}d\phi \,.
\end{aligned}
\end{equation}
In particular, $\Phi(\,\cdot\, |\,0) \equiv 1$.
The following is immediate from Proposition \ref{pro:polykernel}.

\begin{lem} For any $n \in \N\,$ and $\lambda \in \C$ and $\xi \in \partial\DD\,$,
$$
(\Lap - \lambda \, I)^n \,\Phi_n(z \,|\, \lambda) = \Phi(z \,|\, \lambda)\,.
$$
In the specific case $\lambda = -\frac14$ one even has
$$
(\Lap - \lambda \, I) \,\Phi_n(z\,|\,-\! \tfrac14) = \Phi_{n-1}(z\,|\,-\! \tfrac14)\,.
$$
\end{lem}

\begin{rem}\label{rem:zeroes} For 
$\lambda \in \C \setminus (-\infty\,,\,-\frac14)$, the spherical function
$\Phi(r \,|\, \lambda)$ has no zeroes in $[0\,,\,1)$, while for real 
$\lambda \in (-\infty\,,\,-\frac14)$, 
$\Phi(r \,|\, \lambda)$ has countably many zeroes which accumulate at $1$.

For real $\lambda \in [-\frac14\,,\,\infty)$, this is obvious because the integrand in 
\eqref{eq:spherical} is positive. For the other values of $\lambda$, 
to the best of our knowledge, it seems that these facts are nowhere 
referred to in the relevant literature on the hyperbolic Laplacian and its
spherical functions. Therefore, we add some explanations. 

Expressing $\Lap$ in polar coordinates leads to the differential equation
$$
\frac{(1-r^2)^2}{4}\, \Phi''(r \,|\, \lambda) + 
\frac{(1-r^2)^2}{4r}\, \Phi'(r \,|\, \lambda) 
- \lambda\, \Phi(r \,|\, \lambda)=0
$$
in the variable $r \in [0\,,1)$. We now substitute $z = \dfrac{1+r^2}{1-r^2}$
and write $\Psi(z \,|\, \lambda) = \Phi(r \,|\, \lambda)$. Then the above differential 
equation transforms into 
$$
(z^2-1)\, \Psi''(z \,|\, \lambda) + 
2z\, \Psi'(z \,|\, \lambda) - \lambda \, \Psi(z \,|\, \lambda) =0\,.
$$
One sees that the latter is Legendre's differential equation;  
see {\sc Hobson}~\cite{Hobson} and, in particular, {\sc Hille}~\cite{Hi1919}, 
\cite[equations (1) and (2)]{Hi1923}.
Recall that the Legendre function $P_a$ with parameter $a \in \C$ solves 
$(1-z^2)P_a''(z) - 2z P'_a(z) + a(a+1) P_a(z)=0$, and
$$
P_a(z) = F\left(a+1,-a;1;\;\;\frac{1-z}{2}\right),
$$ 
where $F$ is Gauss' hypergeometric function, 
and in our case, $a(a+1)=\lambda$.
We get  
$$
\Phi(r \,|\, \lambda) = P_{a}(z) \quad \text{with}\quad  a= -\frac12 + \mus(\lambda)\AND
z = \frac{1+r^2}{1-r^2} \in [1\,,\infty).
$$
Compare with \cite[identity (25)]{Ey}. Note that the hyperbolic Laplacian
in \cite{Ey} is 4 times the one we are using here, and that $\Phi(z \,|\, \lambda)$ 
is the function
$\varphi_0(z,\mu)$ of \cite{Ey}, with $\mu= \tfrac12 + \mus(\lambda)$.
Note also that there are several known identities
between hypergeometric functions with different parameters. For example,  
{\sc Grellier and Otal}~\cite{GrOt} use a different version, which coincides with
the one given above, see e.g. {\sc Lebedev}~\cite[p. 200]{Lebedev}.

It follows from old work of  {\sc Mehler}~\cite{Meh} and is explained in 
\cite[pages 27--28, identity (35)]{Hi1919} 
that for $a = -\frac12 + b\,\im$ with $b \ne 0$,
that is, for $\lambda  \in (-\infty\,,\,-\frac14)$, the function $P_a(z)$ 
has countably many zeroes in $[1\,,\infty)$ which are such that the zeroes of 
$\Phi(r\,|\,\lambda)$ accumulate at $1$. On the other hand, it is comprised in 
\cite[Theorem V and the page preceding it]{Hi1923}, 
that with $\Re(a) \ne -1/2$ the function $P_a(z)$ has no
zeroes in $[1\,,\infty)$.\footnote{We thank Jean-Pierre Otal (Toulouse) for pointing 
us to the books
by Lebedev and Hobson, which led us to the PhD thesis of Hille \cite{Hi1919} and its 
follow-up \cite{Hi1923}. We also acknowledge an exchange with Peter~Sj\"ogren (G\"oteborg)
on the question 
how well this issue of zeroes of the spherical functions is known in the community 
 -- apparently not at all.}
\hss \qed
\end{rem}

For any fixed $\lambda \in \C \setminus (-\infty\,,\,-\frac14)$, we we shall need the asymptotic 
behaviour of $\Phi_n(z \,|\, \lambda)$ when
$|z| \to 1$, that is, when 
$$
\rho(z,0) = \log \frac{1+|z|}{1-|z|} \to \infty\,.
$$
In view of Remark \ref{rem:zeroes}, we do not consider real $\lambda \in (-\infty\,,\,-\frac14)$ 
because of the infinity of zeroes of $\Phi(r \,|\, \lambda)$.

In the sequel, we shall adopt the following convention: in the formulas where $0 < r < 1$ 
(resp. $z \in \DD$ with $|z| = r$) appear,  we always use capital $R= \rho(r,0) = \rho(z,0)$, 
and we do the same in case of subscripts like $r_k \leftrightarrow R_k\,$.

\begin{theorem}\label{thm:asymp} We have the following, as
$R =  \rho(z,0) \to \infty\,$.
\\[3pt]
{\bf (A)\;} If $\lambda \in \C \setminus (-\infty\,,-\frac14\,]$ then
$$
\Phi_n(z \,|\, \lambda) \sim \frac{c(\lambda)}{n!\,\bigl(2\mus(\lambda)\bigr)^n}\, R^n \, 
\exp\Bigl( \bigl( \mus(\lambda) - 1/2 \bigr)R\Bigr)\,, 
$$
where $c(\lambda) \ne 0$. 
\\[3pt]
{\bf (B)\;} If $\lambda =-\frac14$ then
$$
\Phi_n(z \,|\, \lambda) 
\sim  \frac{2}{(2n+1)!\,\pi}\, R^{2n+1} \, \exp( -R/2 )\,. 
$$
In particular, for every $n \geqslant 1$ and 
$\lambda \in \C\setminus (-\infty\,,\,-\frac14)$ there is $0 < r_{n,\lambda} < 1$ such that
$$
\Phi_n(z \,|\, \lambda) \ne 0 \quad \text{for all }\; z \in \DD\;\text{ with }\;|z| 
\geqslant r_{n,\lambda}\,.
$$
For $n=0$, we  set $r_{0,\lambda}=0$.
\end{theorem}

\begin{proof} 
For the parameter $0<r < 1$, the even function 
\begin{equation}\label{eq:Prphi}
P_r(\phi) = P(r,e^{\im\phi}) = \frac{1-r^2}{1+r^2 - 2r\cos \phi}\,,\quad 
\phi \in [-\pi\,,\pi]
\end{equation}
is strictly decreasing in $\phi \in [0\,,\pi]$. It attains its maximum
in $0$ with value $P_r(0)= (1+r)/(1-r) =  e^R$.  We  write 
\begin{equation}\label{eq:taur}
\frac{P_r(\phi)}{P_r(0)} = \frac{1}{1 +  \tau^2 \sin^2(\phi/2)} \quad \text{with}\quad
\tau = \tau_r = \frac{2\sqrt r}{1-r} \sim P_r(0)= e^R  \;\text{ as }\; r \to 1\,.
\end{equation}
\\[3pt]
\emph{Proof of} {\bf (A)}. In this case, $\Re(\mus) > 0$, where $\mus = \mus(\lambda)$. Then
$$
\begin{aligned}
\frac{n!\,(2\mus)^n\,\Phi_n(r \,|\, \lambda)} {e^{(\mus + 1/2)R}\, R^n}
  &= \frac{1}{2\pi} \int_{-\pi}^{\pi} 
      \frac{P_r(\phi)^{\mus + 1/2} \, \bigl(\log P_r(\phi) \bigr)^n}
           {P_r(0)^{\mus + 1/2} \, \bigl( \log P_r(0) \bigr)^n} \, d\phi \\             
           &= \frac{2}{\tau \pi} \int_{0}^{\pi/2} 
           \frac{\tau}{(1 + \tau^2\sin^2 \phi)^{\mus + 1/2}}
      \biggl(\frac{R - \log\bigl(1 + \tau^2 \sin^2 \phi\bigr)}{R} \biggr)^n 
 \, d\phi\,. 
\end{aligned}
$$
We choose $0 < a < \dfrac{2\Re (\mus)}{2\Re (\mus)+1}$ and decompose the last integral
into the two parts
\begin{center}$\displaystyle \int_0^{\tau^{-a}}$ plus $\displaystyle \int_{\tau^{-a}}^{\pi/2}$.
\end{center} 
As for the second integral,
\begin{equation}\label{eq:second}
\biggl|\frac{R - \log\bigl(1 + \tau^2 \sin^2 \phi\bigr)}{R} \biggr| \leqslant 1
 \end{equation}
Furthermore,
\begin{equation}\label{eq:tautoinfty}
\begin{aligned}
(\tau^2 \sin^2 \phi)^{\Re(\mus)+1/2}  &\geqslant (\tau^2 \sin^2 \tau)^{\Re(\mus)+1/2}\\ 
&\sim \tau^{(1-a)(2\Re(\mus)+1)}
\quad \text{for }\;\phi \in [\tau^{-a}\,,\pi/2]\; \text{ as }\; \tau \to \infty\,,
\end{aligned}
\end{equation}
and we see that the second integral tends to $0$ by the choice of $a$. 

As for the first integral in the decomposition, 
as $\tau$, hence $R$, tend to infinity, we have
\begin{equation}\label{eq:sin}
(1 + \tau^2\sin^2 \phi)^{\mus + 1/2} \sim (1 + \tau^2\phi^2)^{\mus + 1/2}  \AND 
\frac{R - \log\bigl(1 + \tau^2 \sin^2 \phi\bigr)}{R} \sim 1- \frac{1 + \tau^2\phi^2}{\log \tau}\,.
\end{equation}
Therefore, with the substitution $x = \tau\phi$, by dominated convergence as 
$\tau \to \infty$, we obtain
\begin{align}\label{eq:clambda}
\frac{2}{\pi}\int_0^{\tau^{-a}} \frac{\tau}{(1 + \tau^2\sin^2 \phi)^{\mus + 1/2}}
 &\biggl(\frac{R - \log\bigl(1 + \tau^2 \sin^2 \phi\bigr)}{R} \biggr)^n 
 \nonumber \\
 &\sim  \frac{2}{\pi}\int_0^{\tau^{-a}} \frac{\tau}{(1 + \tau^2\phi^2)^{\mus + 1/2}}
 \biggl(1-\frac{\log\bigl(1 + \tau^2 \phi^2\bigr)}{\log \tau} \biggr)^n  \, d\phi 
 \nonumber \\
 &= \frac{2}{\pi}\int_0^{\tau^{1-a}} \frac{1}{(1 + x^2)^{\mus + 1/2}}
 \biggl(1-\frac{\log\bigl(1 + x^2\bigr)}{\log\tau}\biggr)^n \, dx \nonumber \\ 
 &\to
 \frac{2}{\pi}\int_0^{\infty} \frac{1}{(1 + x^2)^{\mus + 1/2}}
 \, dx = c(\lambda),
 \end{align}
recalling that $\mus = \mus(\lambda)$. This proves \textbf{(A)}.
\\[6pt]
\emph{Proof of} {\bf (B)}. In this case, $\mus = 0$, and
\begin{equation*}
\begin{aligned}
\frac{(2n)!\,\Phi_n(r|-\tfrac{1}{4})} {e^{R/2}\, R^{2n}}
           &\sim \frac{2}{\pi}\, \int_{0}^{\pi/2} \frac{1}{(1 + \tau^2\sin^2 \phi)^{1/2}}
 \biggl(1 - \frac{\log\bigl(1 + \tau^2 \sin^2 \phi\bigr)}{\log \tau}\biggr)^{2n} 
 \, d\phi\,. 
\end{aligned}
\end{equation*}
This time, we decompose the last integral into the two parts
\begin{center}
$\displaystyle \int_0^{1/\log\tau}$ plus $\displaystyle \int_{1/\log\tau}^{\pi/2}$. 
\end{center}
By using the analogue of \eqref{eq:sin},  the first integral is asymptotically
equivalent to 
$$
\int_0^{1/\log\tau} \frac{1}{(1 + \tau^2\,\phi^2)^{1/2}}
 \biggl(1 - \frac{\log\bigl(1 + \tau^2\,\phi^2\bigr)}{\log \tau}  \biggr)^{2n}\,d\phi \quad 
 \text{as }\; \tau \to \infty\,.
$$
We now substitute  $x = \dfrac{\log\bigl(1 + \tau^2\,\phi^2\bigr)}{\log \tau}$, and
observe that the upper integration limit $1/\log \tau$ transforms into 
$$
b_{\tau} = \frac{\log\bigl(1 + (\tau/\log \tau)^2\bigr)}{\log \tau} \sim 2\,,\quad 
\text{as }\; \tau \to \infty\,.
$$
Thus, the latter integral becomes
$$
\frac{\log \tau}{2\tau} \int_0^{b_{\tau}} (1-x)^{2n} \frac{dx}{\bigl(1- \tau^{-x}\bigr)^{1/2}}
\sim \frac{\log \tau}{\tau} \cdot \frac{1}{2n+1}\,.
$$
By \eqref{eq:second}, the second integral is bounded by
$$
\int_{1/\log \tau}^{\pi/2} \frac{d\phi}{\tau \sin\phi} 
=  \frac{1}{\tau}\Bigl(- \log\, \tan \frac{1}{2\log\tau} \Bigr) 
\sim \frac{\log \tau}{\tau}  \cdot \frac{\log(2\log\tau)}{\log \tau}
\,,\quad 
\text{as }\; \tau \to \infty\,.
$$
This leads to the asymptotic behaviour of \textbf{(B)}.
\end{proof}

For $\lambda \in \C \setminus (-\infty\,,-\frac14\,]$, we define the \emph{associated
real eigenvalue} $\lambda^* \in [-\frac14\,,+\infty)$ by the equation
\begin{equation}\label{eq:assoc-lambda}
\mus(\lambda^*) = \Re \bigl(\mus(\lambda)\bigr). 
\end{equation}

We shall also need the function defined as
\begin{equation}\label{eq:absPhi}
|\Phi|_n(z \,|\, \lambda) =\int_{\partial \DD}  \bigl|P_n(z,\xi \,|\, \lambda)\bigr|\,d\xi\,, \quad z \in \DD\,.
\end{equation}
The same computations as in the proof of Theorem \ref{thm:asymp} with this 
modified integrand lead to the following.
\begin{pro}\label{pro:absasymp} 
We have the following. \\[3pt]
{\bf (A)\;} If $\lambda \in \C \setminus (-\infty\,,-\frac14\,]$ and $\mus = \mus(\lambda)$ then,
as $R =  \rho(z,0) \to \infty\,$
$$
|\Phi|_n(z \,|\, \lambda) 
\sim C_n(\lambda) \, R^n \, \exp\Bigl( \bigl( \Re(\mus) - 1/2 \bigr)R\Bigr)\,,$$
as $\;R =  \rho(z,0) \to \infty\,$, where 
$\; C_n(\lambda) = \dfrac{c(\lambda^*)}{n!\,(2|\mus|)^n}\;$
with $c(\lambda^*)$ according to \eqref{eq:clambda}. 
\\[3pt]
{\bf (B)\;} If $\lambda =-\frac14$ then 
$\;
|\Phi|_n(z \,|\, \lambda) = \Phi_n(z \,|\, \lambda)\;$
with asymptotics given by Theorem \ref{thm:asymp}.B.
\end{pro}

Regarding the zeroes of the higher order polyspherical functions,
we have $P_n(0,\xi \,|\, \lambda) = \Phi_n(0\,|\,\lambda) = 0$ for all $\lambda \in \C$ and 
$n \geqslant 1$. For the following, recall that
$$
J_n = \frac{1}{2\pi} \int_{-\pi}^{\pi} (\cos \phi)^n \, d\phi =
\begin{cases}  \displaystyle 
               \frac{1}{2^n}{n \choose n/2}\,,&\text{if $n$ is even,}\\[.4cm]
               0\,,&\text{if $n$ is odd.}              
\end{cases}
$$
\begin{lem}\label{lem:at0} Let $n \geqslant 1$. Then we have the following as $r \to 0$.
\\[3pt]
{\bf (A)\;} If $\lambda \in \C \setminus \{-\frac14\}$ and $\mus = \mus(\lambda)$, 
$$
\Phi_n(r \,|\, \lambda) \sim 
\begin{cases} \dfrac{1}{\bigl(\frac{n}{2}!\bigr)^2\,(2\mus)^n}\,r^n,
              &\text{if $n$ is even,}\\[15pt] 
              \dfrac{1}{\frac{n+1}{2}! \,\frac{n-1}{2}! \,(2\mus)^{n-1}}\,r^{n+1},
              &\text{if $n$ is odd.}              
\end{cases}
$$
\\[3pt]
{\bf (B)\;} If $\lambda =-\frac14$ then
$$
\Phi_n(r\,|\, \lambda) 
\sim  \frac{1}{(n!)^2}\, r^{2n}.   
$$
\end{lem}

\begin{proof}
In the sequel all $\mathfrak{o}(r^k)$ are uniform in $\phi \in (-\pi\,,\,\pi]$,
as $r \to 0$. Short computations show that
$$
\begin{aligned} 
P(r,e^{\im \phi}) &= 1 + 2r\,\cos\phi + \mathfrak{o}(r) \AND\\
\log P(r,e^{\im \phi}) &= 2r\,\cos\phi + 2r^2 (\cos^2 \phi-1) + \mathfrak{o}(r^2)\,, 
\end{aligned}
$$
so that 
\begin{equation}\label{eq:expansion}
\begin{aligned} 
&P(r,e^{\im \phi})^{\mus + 1/2}\, \bigl(\log P(r,e^{\im \phi})\bigr)^n\\ 
&\qquad= \Bigl(1 + (2\mus +1)\,r\,\cos\phi + \mathfrak{o}(r)\Bigr)\\
&\qquad\quad \times 
\Bigl( 2^nr^n\,\cos^n\phi + n 2^{n+1}r^{n+1} (\cos^{n+1}\phi - \cos^{n-1} \phi)\Bigr) 
+ \mathfrak{o}(r^{n+1})\\
&\qquad= 2^nr^n\,\cos^n\phi + 2^nr^{n+1}\Bigl( (2\mus +1 +n) \cos^{n+1}\phi - n \,\cos^{n-1} \phi \Bigr)
+ \mathfrak{o}(r^{n+1})
\end{aligned}
\end{equation}
Applying $\frac{1}{2\pi} \int_{-\pi}^{\pi}$ with respect to $\phi$, using the above 
formula for $J_n\,$, and 
normalising as in  Proposition \ref{pro:polykernel}, we obtain 
the proposed asymptotic behaviour near~$0$.
\end{proof}

While we do not know much about the zeroes of $\Phi_n(r \,|\, \lambda)$ in general,
we have the following for real eigenvalues.

\begin{pro}\label{pro:positive}
For real $\lambda \geqslant -\frac14$ and all $n \in \N$, we have $\Phi_n(r \,|\, \lambda) > 0$
for $0 < r < 1$.
\end{pro}

\begin{proof} The statement is clear for even $n$, as well as for all $n$ when $\lambda = -\frac14$.
Fix $r > 0$ and consider 
$$
F_n(\mus) = \int_{\partial \DD} \bigl(\log P(r,\xi)\bigr)^n \, P(r,\xi)^{\mus +1/2}\, d\xi\,, 
$$
a function of $\mus \geqslant 0$.
For $\mus = \mus(\lambda)$, one has $\Phi_n(r \,|\, \lambda) = F_n(\mus)$ when 
$\lambda > -\frac14$. Then for $\mus > 0$
$$
\frac{d}{d\mus} F_{2n+1}(\mus) =  F_{2n+2}(\mus) > 0\,,
$$
so that $F_{2n+1}(\mus)$ is strictly increasing in $\mus \in [0\,,\,\infty)$.
To complete the proof it is enough to show that $F_{2n+1}(0)=0$, or equivalently,
that
$$
f_{2n+1}(z) = \int_{\partial \DD} \bigl(\log P(z,\xi)\bigr)^{2n+1} \, \sqrt{P(z,\xi)} \,d\xi = 0 \quad 
\text{for all }\; z \in \DD\,.
$$
We show this by induction on $n$. For $n =1$, Lemma \ref{lem:diff} yields that
$(\Lap + \frac14 I) f_1 = 0$ (recalling that $\mus =0$).
Since $f_1$ is radial, i.e., it depends only on $r = |z|$, it must be a constant multiple
of the spherical function $\Phi(z \,|\, -\frac14)$. Moreover, as $f_1(0) =0$, that 
constant factor must be $0$, that is, $f_1 \equiv 0$ on $\DD$. Now suppose that 
$n \geqslant 1$ and that we have already
that $f_{2n-1} \equiv 0$ on $\DD$. Once more, from Lemma \ref{lem:diff},
we get
$$
(\Lap + \tfrac14 I) f_{2n+1} = 2n(2n+1)\, f_{2n-1} \equiv 0\,
$$
so that also $f_{2n+1}$ must be a constant multiple of $\Phi(z \,|\, -\frac14)$. As above, we get
$f_{2n+1} \equiv 0$ on $\DD$.
\end{proof}

For $\lambda \in \C \setminus \R$, from Theorem \ref{thm:asymp} we know
at least that $\Phi_n(z \,|\, \lambda) \ne 0$ when $|z|$ is sufficiently close to $1$.

\medskip

\section{Polyharmonic kernels}\label{sec:kernels}

\begin{dfn}\label{def:Poistr}
Let $n \in \N_0\,$ and $\lambda \in \C$. 
The \emph{$n^{\text{th}}$ generalised Poisson transform} of an analytic functional (or measure) 
$\nu$ on $\partial \DD$
is the function 
$$
\Poiss_{n,\lambda}\nu(z) = \int_{\partial \DD}  P_n(z,\xi \,|\, \lambda)\,d\nu(\xi)\,,
\quad z \in \DD\,,
$$
where $P_n$ is given by Proposition \ref{pro:polykernel}.
If $\nu$ is an absolutely continuous measure with density function $g(\xi)$ with respect to 
the normalised Lebesgue measure $\leb$ on the circle,
then we write 
$\Poiss_{n,\lambda}g(z)$ for the resulting transform. When $n=0$, we just write 
$$
\Poiss_{\lambda} = \Poiss_{0,\lambda} \AND \Poiss = \Poiss_{0,0}\,.
$$
\end{dfn}

Our aim is to consider boundary value problems for polyharmonic functions.
The most basic one is the Dirichlet
problem: given a continuous function $g$ on $\partial \DD$, look for a harmonic function
$h$ on $\DD$ which provides a continuous extension of $g$ to the closed disk $\DD \cup \partial \DD$,
the hyperbolic compactification.  The solution is unique and well known, 
$h(z) = \Poiss g(z)\,$, the classical Poisson transform.

If we consider $\lambda$-harmonic functions with 
$0 \ne \lambda \in \C\setminus (-\infty\,,\,-\frac14)$, then  we cannot proceed in 
the same way, using $P(z,\xi)^{\mus(\lambda) + 1/2}$. Indeed, if we take 
$g \equiv 1$ on $\partial \DD$, then its $\lambda$-Poisson transform is
$h(z) = \Phi(z \,|\, \lambda)$, 
which tends to $0$ as $|z| \to 1$ when $\mus =0$ or $0 < \Re(\mus) < 1/2$, and to 
$\infty$ in absolute value when $\Re(\mus) > 1/2$.
It is well-known that in this case one should normalise, and the natural candidate is 
$\Phi(z \,|\, \lambda)$,
see e.g.~\cite{Mich} and \cite{KP}. (Note that this normalisation is not feasible when 
$\lambda \in (-\infty\,,\,-\frac14)$ because of the zeroes of the spherical function 
which accumulate at $\partial \DD$.)

\begin{dfn}\label{def:normalizedkernel}
For $n \in \N_0\,$ and $\lambda \in \C \setminus (-\infty\,,\,-\frac14)$, the 
\emph{normalised polyharmonic kernel} is
$$
\mathcal{K}_{n,\lambda}(z,\xi) = \frac{P_n(z,\xi \,|\, \lambda)}{\Phi_n(z \,|\, \lambda)}\,,
\quad \xi \in \partial \DD\,, \; z \in \DD\,,\;|z| \geqslant r_{n,\lambda}\,,
$$
where $r_{n,\lambda}$ is as in Theorem \ref{thm:asymp}.
\end{dfn}

Thus, 
$$
\int_{\partial \DD} \mathcal{K}_{n,\lambda}(z,\xi)\, d\nu(\xi) 
= \frac{\Poiss_{n,\lambda}\nu(z)}{\Phi_n(z \,|\, \lambda)}\,.
$$

\begin{lem}\label{lem:Kcompare}
There is a constant $\widetilde C(\lambda)$ for which the kernels  
of Definition \ref{def:normalizedkernel} satisfy
$$
\bigl|\mathcal{K}_{n,\lambda}(z,\xi)\bigr| 
\leqslant \widetilde C(\lambda) \, \mathcal{K}_{0,\lambda^*}(z,\xi) 
$$
where $\lambda^* \geqslant -\frac14$ is given by \eqref{eq:assoc-lambda} 
and $|z| \geqslant r_{n,\lambda}\,$.
\end{lem}

\begin{proof}
Note that $|\hor(re^{\im\alpha},\xi)| \leqslant \hor(r,1) = R$, where (recall) $r=|z|$ and
$R = \rho(z,0) = \rho(r,0) = \log \dfrac{1+r}{1-r}$. By Theorem \ref{thm:asymp},
as $r=|z|\to 1$,
$$
\bigl|\Phi(z \,|\, \lambda)\bigr| \sim
\frac{|c(\lambda)|}{c(\lambda^*)}\, \Phi(z \,|\, \lambda^*). 
$$
Note also that by Theorem \ref{thm:asymp},
$$
\bigl|\Phi_n(z \,|\, \lambda)\bigr| \sim 
\begin{cases} \bigl|\Phi(z \,|\, \lambda)\bigr|\, R^n \, \dfrac{1}{n! \bigl(2|\mus(\lambda)|\bigr)^n}\,,
&\text{if }\;\lambda \in \C\setminus (-\infty\,,\,-\frac14)\,,\\[12pt]
\bigl|\Phi(z \,|\, \lambda)\bigr|\, R^{2n} \, \dfrac{1}{(2n+1)!} \,,
&\text{if }\;\lambda =-\frac14 \; (= \lambda^*)\,,
\end{cases}
$$ 
as $r \to 1$. 
Therefore, there is $\widetilde C(\lambda) > 0$ (depending on $n$) such that 
$$
\widetilde C(\lambda) \bigl|\Phi_n(z \,|\, \lambda)\bigr| \geqslant
\begin{cases} \Phi(z \,|\, \lambda^*)\, R^n \,, 
&\text{if }\;\lambda \in \C\setminus (-\infty\,,\,-\frac14)\,,\\[12pt]
\Phi(z \,|\, \lambda^*)\, R^{2n} \,,
&\text{if }\;\lambda =-\frac14\,,
\end{cases}
$$ 
for all $r \geqslant r_{n,\lambda}\,$. 
It follows that $\bigl|\mathcal{K}_{n,\lambda}(z,\xi)\bigr|$ is bounded above by 
\begin{align*}
\frac{P(z,\xi)^{\Re(\mus(\lambda))+1/2}R^n}{\bigl|\Phi_n(z \,|\, \lambda)\bigr|} 
&\leqslant 
\widetilde C(\lambda)\,\frac{P(z,\xi)^{\mus(\lambda^*)+1/2}}{\Phi(z \,|\, \lambda^*)},\quad \text{if }\; 
\lambda \in \C\setminus (-\infty\,,\,-\frac14\,]\,,\\
\intertext{and by}
\frac{P(z,\xi)^{1/2}R^{2n}}{\Phi_n(z|-\tfrac14)
} &\leqslant 
\widetilde C(-\tfrac14)\,\frac{P(z,\xi)^{1/2}}{\Phi(z|-\tfrac14)},\quad \text{if }\; 
\lambda  = -\frac14\,,
\end{align*}
which proves our claim.
\end{proof}

\begin{rem}\label{rem:well}
(a) If $\lambda > -\frac14$ then, by Lemma \ref{lem:at0}, 
$\mathcal{K}_{n,\lambda}(z,\xi)$ is well-defined for all $z \in \DD \setminus \{0\}$. 
When $n$ is even then, in view of \eqref{eq:expansion}, it extends continuously to $z=0$,
and we can set $r_{n,\lambda}=0$ for even $n$. 

However, when $n$ is odd, again in view of \eqref{eq:expansion},
$\mathcal{K}_{n,\lambda}(z,\xi)$ has a pole at $z=0$ unless $\xi = \pm \im$.
Thus, for Lemma \ref{lem:Kcompare} we need to work with $r_{n,\lambda}=\ep$ for odd $n$, 
where $\ep > 0$ is arbitrary but fixed.
\\[3pt]
(b) If $\lambda = -\frac14$ then $\mathcal{K}_{n,\lambda}(z,\xi)$ extends continuosly to $z=0$
for all $n$, and we can always use $r_{n,-\frac14} = 0$.
\end{rem}

Note that $\mathcal{K}_{0,\lambda^*}(z,\cdot)\, d\leb$ is a probability
measure on $\partial \DD$ for each $z \in \DD$. 

Another look at the proof of Theorem \ref{thm:asymp} yields the following.

\begin{lem}\label{lem:tozero} Let $\lambda \in \C \setminus (-\infty\,,\,-\frac14)$, 
$\mus = \mus(\lambda)$ and $n \in \N_0\,$. Then 
$$
\begin{aligned}
&\lim_{r \to 1} \mathcal{K}_{n,\lambda}(r,e^{\im\psi}) = 0
\qquad \text{uniformly for}\\[4pt] 
|\psi| \in &\begin{cases}  
[2\tau^{-a}\,,\,\pi] \,,
&\text{if }\; \lambda \in \C \setminus (-\infty\,,\,-\frac14)\,,\quad\text{where }\; 
0 < a < \dfrac{2\Re (\mus)}{2\Re (\mus)+1}\,,\\
[2(\log \tau)^{-a}\,,\,\pi]\,,&\text{if }\; \lambda = -\frac14\,,\quad\text{where }\; 
0 < a < 1\,.
\end{cases}
\end{aligned}
$$
\end{lem}

\begin{proof} In view of Lemma \ref{lem:Kcompare}, we only need to prove this 
for $n=0$ and real $\lambda \geqslant -\frac14$. In this case, it is well known except maybe for the fact that
usually the lower bound for $|\psi|$ is required to be a positive constant, while here it 
tends to $0$ as $r \to 1$.

\smallskip

First, look at case \textbf{(A)} of Theorem \ref{thm:asymp}.
With $\tau$ as in \eqref{eq:taur} and $c(\lambda)$ as in \eqref{eq:clambda}, 
we have as $r \to 1$ 
$$
\frac{P(r,e^{\im\psi}|\lambda)}{\Phi(r \,|\, \lambda)}
\sim \frac{P_r(\psi)^{\mus + 1/2}}
           {c(\lambda)\,P_r(0)^{\mus + 1/2} \,  e^{-R}}
\sim \frac{1}{c(\lambda)} \,  \frac{\tau}{\bigl(1 + \tau^2\sin^2 (\psi/2)\bigr)^{\mus + 1/2}}.
$$
By \eqref{eq:tautoinfty} and \eqref{eq:sin},
this tends to $0$ uniformly in the stated range.

In case \textbf{(B)} of Theorem \ref{thm:asymp}, 
as $r\to 1$, that is, $\tau\to\infty$,
$$
\frac{P(r,e^{\im\psi}|-\frac14)}{\Phi(r \,|\, -\frac14)}
\sim \frac{\pi}{2}\,\frac{P_r(\psi)^{1/2}}
           {R\, e^{-R} \,P_r(0)^{1/2}}
\sim \frac{\pi}{2}\,\frac{\tau/\log \tau}{\bigl(1 + \tau^2\sin^2 (\psi/2)\bigr)^{1/2}}.
$$
Again, this tends to $0$ uniformly in the stated range.
\end{proof}

The kernels are also rotation invariant:
$\mathcal{K}_{n,\lambda}(e^{\im \alpha}z,e^{\im \alpha}\xi) = 
\mathcal{K}_{n,\lambda}(z,\xi)$.
This fact and the last two lemmas yield the following by 
well-known methods.

\begin{pro}\label{pro:basic_convergence}
Let $n \in \N_0\,$ and $\lambda \in \C \setminus (-\infty\,,\,-\frac14)$. For a 
measurable function $g : \partial \DD \to \C$, resp. a complex Borel measure $\nu$
on $\partial \DD$, let
$$
f(z) = \Poiss_{n,\lambda}g(z)
\,,\quad \text{resp.}\quad 
f(z) = \Poiss_{n,\lambda}\nu(z)\,.
$$
Then the following properties hold.
\begin{enumerate}
 \item[(i)] If $g \in C(\partial \DD)$ then 
 $\;\;\displaystyle \lim_{z \to \xi} \frac{f(z)}{\Phi_n(z \,|\, \lambda)} = g(\xi)\;\;$ for all 
 $\xi \in \partial \DD$ and uniformly as $|z| \to 1$.
\item[(ii)] If $g \in L^p(\partial \DD)$ ($1 \leqslant p < \infty$) then 
  $\;\;\displaystyle \lim_{r \to 1} \frac{f(r\xi)}{\Phi_n(r \,|\, \lambda)} = g(\xi)\;\;$ for almost every
  $\xi \in \partial \DD$ and in $L^p(\partial \DD)$.
\item[(iii)] If $g \in L^{\infty}(\partial \DD)$ then 
  $\;\;\displaystyle \lim_{r \to 1} \frac{f(r\xi)}{\Phi_n(r \,|\, \lambda)} = g(\xi)\;\;$  in the weak*-topology of 
  $L^{\infty}(\partial \DD)$.  
\item[(iv)] If $\nu$ is a finite Borel measure on $\partial \DD$ then the measures
$\;\;\dfrac{f(r\xi)}{\Phi_n(r \,|\, \lambda)}\,d\xi\;\;$ 
converge to $\nu$ in the weak*-topology, as $r \to 1$. 
\end{enumerate}
\end{pro}

The classical reference is
{\sc Zygmund} \cite[chapter 17, Theorems 1.20 \& 1.23]{Zygmund}.
See also {\sc Stein and Weiss} \cite[\S I.1, Theorems 1.18 \& 1.25 and \S II.2]{StWe}
and {\sc Garnett} \cite[Theorem 3.1]{Gar}. In these references, the results are 
presented for the upper half space, resp. half plane and carry over to the disk model 
of the hyperbolic plane. 

\medskip

\section{Dirichlet and Riquier problem at infinity}\label{sec:DirRiq}

Reconsider statement (i) of Proposition \ref{pro:basic_convergence}. It says that for 
any $g \in \mathcal{C}(\partial \DD)$, the function 
\begin{equation}\label{eq:divide}
 f(z) = \frac{\Poiss_{n,\lambda}g(z)}{\Phi_n(z \,|\, \lambda)}
\end{equation}
is $\lambda$-polyharmonic of order $n+1$ which provides a continuous extension of
$g$ to $\{ z \in \DDbar: |z| \geqslant r_{n,\lambda} \}\,$.  
When $n \geqslant 1$ in \eqref{eq:divide} then we cannot expect that the given $f$ is the unique
function with this property. Indeed, for example also 
$$
\lim_{z \to \xi} \frac{f(z)+ \Phi_k(z \,|\, \lambda)}{\Phi_n(z \,|\, \lambda)} = g(\xi) \quad 
\text{when }\; 0 \leqslant k < n\,.
$$
However, when $n=0$ and we are considering $\lambda$-harmonic functions, this is the solution 
of the $\lambda$-Dirichlet problem, valid on all of $\DDbar$ since 
$\Phi(\cdot \,|\, \lambda)$ has no zeroes in $\DD$. It is well-known to be unique 
when $\lambda =0$. The extension 
to \emph{real} $\lambda \geqslant -\frac14$ with normalisation by $\Phi(z \,|\, \lambda)$ is 
well understood 
via the maximum principle applied to the kernels $\mathcal{K}_{0,\lambda}(z,\xi)$ of 
Definition \ref{def:normalizedkernel}, which are 
probability kernels with respect to $\leb$ for real $\lambda$. 
When $\lambda \in \C\setminus (-\infty\,,\,-\frac14)$ is \emph{complex}, 
we can still prove uniqueness,
via a different technique in place of the standard one.

\begin{theorem}\label{thm:dirichlet}
Let $\lambda \in \C\setminus (-\infty\,,\,-\frac14)$ and
 $g \in \mathcal{C}(\partial \DD)$. Then
 $h(z) = \Poiss_{\lambda}g(z)$ 
 is the \emph{unique} $\lambda$-harmonic function for which 
\begin{equation}\label{eq:div}
 \lim_{\DD \ni z \to \xi} \frac{h(z)}{\Phi(z \,|\, \lambda)} = g(\xi) 
 \quad \text{for every }\; \xi \in \partial \DD. 
\end{equation}
\end{theorem}

\begin{proof}
We  know that the given function $h$ is a solution of the $\lambda$-Dirichlet 
problem, i.e. it satisfies \eqref{eq:div}. In order to show uniqueness, it is enough to 
prove that the constant function $0$ is the only solution, when $g \equiv 0$ on $\partial \DD$. 
In other words, we assume that $h$ is a $\lambda$-harmonic function on $\DD$ with
$$
 \lim_{\DD \ni z \to \xi} \frac{h(z)}{\Phi(z \,|\, \lambda)} = 0
 \quad \text{for every }\; \xi \in \partial \DD\,, 
$$
and we have to show that $h \equiv 0$.

Let $\dol h$ be the spherical average of $h$ around $0$, 
$$
\dol h(z) = \frac{1}{2\pi} \int_{-\pi}^{\pi} h(|z|e^{\im\phi})\, d\phi\,. 
$$
In particular, $\dol h(0)=h(0)$.
Then $\dol h$ is also $\lambda$-harmonic, and it is
rotation-invariant. Now, up to multiplication with constants, $\Phi(z \,|\, \lambda)$
is the unique $\lambda$-harmonic function which is rotation-invariant. See e.g. 
\cite{Ey} for this well-known fact. Therefore
$$
\dol h(z) = h(0) \, \Phi(z \,|\, \lambda)\,.
$$
On the other hand, the function on the closed disk which is $h/\Phi(\cdot\,|\lambda)$ in the interior 
and~$0$ on the boundary is continuous, whence uniformly continuous, so that 
$$
\lim_{|z| \to 1} \frac{h(z)}{\Phi(z \,|\, \lambda)} = 0
\quad \text{uniformly in }\; z.
$$
Since $\Phi(z \,|\, \lambda)$ only depends on $|z|$, also 
$$
\lim_{|z| \to 1} \frac{\dol h(z)}{\Phi(z \,|\, \lambda)} = 0.
$$
We conclude that $h(0)=0$. 

Now let $z_0 \in \DD$ be arbitrary. Then there is an isometry $\gamma$
of the Poincar\'e disk (a M\"obius transform) such that $\gamma\, 0 = z_0\,$.
The isometries commute with the hyperbolic Laplacian, whence also the function
$h_{\gamma}(z) = h(\gamma z)$ is $\lambda$-harmonic. If $|z| \to 1$ then also $|\gamma z| \to 1$.
Therefore also $h(\gamma z)/\Phi(\gamma z \,|\, \lambda) \to 0$ as $|z| \to 1$. Now 
Theorem \ref{thm:asymp} implies 
$$
\frac{\Phi(\gamma z \,|\, \lambda)}{\Phi(z \,|\, \lambda)} \sim 
\exp \Bigl((\mus-1/2)\bigl(\rho(\gamma\, z,o) - \rho(z,o)\bigr)\Bigr) \quad \text{as }\; |z| \to 1\,.
$$
This is bounded, since
$$
\bigl|\rho(\gamma\, z,o) - \rho(z,o)\bigr| = \bigl|\rho(z,\gamma^{-1}o) - \rho(z,o)\bigr| \leqslant
\rho(\gamma^{-1}o,o).
$$
We infer that
$$
\frac{h_{\gamma}(z)}{\Phi(z \,|\, \lambda)} = \frac{h(\gamma z)}{\Phi(\gamma z \,|\, \lambda)} \,
\frac{\Phi(\gamma z \,|\, \lambda)}{\Phi(z \,|\, \lambda)} \to 0 \quad \text{as }\; |z| \to 1\,.
$$
We can now apply the above argument to
$h_{\gamma}$ and its spherical average, and conclude that $h(z_0) = 0$. This is true for
every $z_0 \in \DD$.
\end{proof}

The next Lemma is in preparation of the Riquier problem for $\lambda$-polyharmonic functions.

\begin{lem}\label{lem:preRiq}
For $\lambda \in \C \setminus (-\infty\,,\,-\frac14)$, let $f$ be $\lambda$-polyharmonic of order $n+1$ 
on $\DD$ and such that the $\lambda$-harmonic function $h =(\Lap - \lambda \, I)^n f$
satisfies 
$$
\lim_{\DD \ni z \to \xi} \frac{h(z)}{\Phi(z \,|\, \lambda)} = g(\xi) 
 \quad \text{for every }\; \xi \in \partial \DD\,
$$
where $g \in \mathcal{C}(\partial \DD)$. Then
$$
f(z) = \Poiss_{n,\lambda}g(z) + f_*(z)\,,
$$
where $f_*$ is $\lambda$-polyharmonic of order $n$.
\end{lem}

\begin{proof} This is very similar to the inductive argument in the proof of Theorem \ref{thm:poly}.
It follows from Theorem \ref{thm:dirichlet} that $h = \Poiss_{\lambda}g =: h_g\,$.
Write $f_g = \Poiss_{n,\lambda}g\,$. By Lemma \ref{lem:preRiq}, 
$$
(\Lap - \lambda \, I)^n f_g = h_g = (\Lap - \lambda \, I)^n f\,.
$$
Therefore $(\Lap - \lambda \, I)^n f_* = 0$. 
\end{proof}

In the setting of the Euclidean Laplacian $\Delta$ on a bounded domain, 
the Riquier problem asks for solutions of 
$\Delta^n f = 0$ with prescribed boundary data 
$g_k = \Delta^k f$ for $k=0,\dots,n-1$. This is not applicable to 
the hyperbolic setting of $(\Lap - \lambda \, I)^{n}$, even when $\lambda=0$, because 
in any case, the quotient of lower and higher order polyspherical functions tends to zero
at the boundary. For this reason, we propose a different formulation.

\begin{dfn}\label{def:Riq}
 Let $\lambda \in \C\setminus (-\infty\,,\,-\frac14)$ and 
$g_0\,, \dots\,, g_{n-1} \in \mathcal{C}(\partial \DD)$. Then a solution of
the associated \emph{Riquier problem at infinity} is a polyharmonic function
$$
f = f_0 + \dots + f_{n-1}
$$
of order $n$, where each $f_k$ is $\lambda$-polyharmonic of order $k+1$ and
$$
\lim_{z \to \xi}  \frac{(\Lap -\lambda \, I)^k f_k(z)}{\Phi(z \,|\, \lambda)} = g_k(\xi)  
\quad \text{for every }\; 
\xi \in \partial \DD \,.
$$
\end{dfn}

\smallskip

Note that by Remark \ref{rem:zeroes}, the denominator in the last
quotient is always non-zero.

\begin{cor}\label{cor:Riq}
A solution of the Riquier problem as in Definition \ref{def:Riq} is given by
$$
f_k(z) = \Poiss_{k,\lambda}g_k(z)\,, \quad k = 0, \dots, n-1\,.
$$
One also has for every $\xi \in \partial \DD$ and $k \in \{  0\,,\dots, n-1\}$
$$
\lim_{z \to \xi} \frac{f_k(z)}{\Phi_k(z \,|\, \lambda)} = g_k(\xi) \AND
\lim_{z \to \xi} \frac{f_j(z)}{\Phi_k(z \,|\, \lambda)} = 0 \; \text{ for }\, j < k. 
$$
\end{cor}

\medskip

\section{A Fatou theorem for polyharmonic functions}\label{sec:Fatou}

\begin{dfn}
{\bf (i)} For $0<\delta\leqslant\pi$,  
consider the arc $B_\delta=\{e^{\im\phi}\colon |\phi|<\delta\}\subset\partial \DD\,$, and for 
$\zeta=e^{\im\alpha}$ consider the rotated arc 
$B_\delta(\zeta)=\zeta B_\delta = \{e^{\im\phi}\colon |\phi-\alpha|<\delta\}$ with 
measure (normalised arc length) $\leb\bigl(B_\delta(\zeta)\bigr) = \delta/\pi$. 
The \emph{Hardy--Littlewood maximal operator} is defined on functions $g\in L^1(\partial \DD)$  as
$$
\Mcal g(\zeta)=\sup \left\{ \frac{1}{\leb\bigl(B_\delta(\zeta)\bigr)} 
\int_{B_\delta(\zeta)} |g(\xi)|\,d\xi\;\colon\; 0<\delta \leqslant \pi\right\}.
$$
{\bf (ii)} Let $[0,\zeta]$ be the line segment with endpoints $0$ and  $\zeta=e^{\im\alpha}$ in 
the unit disc $\DD$, and for $a\geqslant 0$ consider the \emph{admissible region}, or 
\emph{tubular domain} 
$$
\Gamma_a(\zeta) = \Bigl\{z\in \DD\colon \rho\bigl(z,[0,\zeta] \bigr) \leqslant a \Bigr\}.
$$
The \emph{non-tangential (tubular) maximal operator of width $a \geqslant0$} is defined on 
functions $g\in L^1(\partial \DD)$  as
$$
\Mfrak_a^{(n,\lambda)} g(\zeta)=\sup   \left\{ \,  \biggl|    
 \int_{\partial \DD} \mathcal{K}_{n,\lambda}(z,\xi)\,g(\xi)\, d\xi \biggr| \;\colon\; 
z\in \Gamma_a(\zeta)\,,\;|z| \geqslant r_{n,\lambda}
\right\},
$$
where $\mathcal{K}_{n,\lambda}$ is the kernel introduced in Definition \ref{def:normalizedkernel}
and $r_{n,\lambda}$ is as in Theorem \ref{thm:asymp}.
\\[4pt]
{\bf (iii)} With the same ingredients as in (ii), the \emph{enlarged-admissible region} is
$$
\Gamma_{(a)}(\zeta) = \Bigl\{z\in \DD\colon \rho\bigl(z,[0,\zeta] \bigr) 
\leqslant a + \log \rho(z,0) \Bigr\}.
$$
The \emph{extended maximal operator of width $a$} is defined on 
functions $g\in L^1(\partial \DD)$  as
$$
\Mfrak_{(a)}^{(n,\lambda)} g(\zeta)=\sup   \left\{ \,  \biggl|    
 \int_{\partial \DD} \mathcal{K}_{n,\lambda}(z,\xi)\,g(\xi)\, d\xi \biggr| \colon 
z\in \Gamma_{(a)}(\zeta)\,,\;|z| \geqslant r_{n,\lambda}
\right\}.
$$
\end{dfn}

Both $\Gamma_a(\zeta)$ and $\Gamma_{(a)}(\zeta)$ are (Euclidean) convex subsets of $\DD$
which touch the boundary  $\partial \DD$ only at $\zeta$. In the Euclidean metric, the 
``tube'' $\Gamma_a(\zeta)$ is conical at $\zeta$  with Stolz angle $2\arctan(\sinh a)$. 
On the other hand, the boundary
curve of the larger domain $\Gamma_{(a)}(\zeta)$ is tangent to  $\partial \DD$ at $\zeta$.

The following fact is well-known (see, for instance, \cite[Corollary of Lemma 1.1]{Koranyi&Vagi}).

\begin{pro}\label{prop:HL_is_weak_type_1-1} The Hardy--Littlewood maximal operator 
is of weak type $(1,1)$ and strong type $(p,p)$ for  $1<p\leqslant\infty$.
\end{pro}

Our aim is to prove the following.

\begin{pro}\label{pro:maximal_functions_inequality}
For every $a\geqslant 0$ and $n\in\N_0$,  
there is a constant $C(a, \lambda)>0$ such that for all $g\in L^1(\partial\DD)$
$$
\begin{aligned}
\Mfrak_a^{(n,\lambda)}\,g &\leqslant C(a, \lambda) \, \Mcal g\,, \quad \text{if }\; 
\lambda \in \C\setminus (-\infty\,,\,-\frac14)\,,\\
\Mfrak_{(a)}^{(n,\lambda)}\,g &\leqslant C(a, \lambda) \, \Mcal g\,, \quad \text{if }\; 
\lambda = -\frac14\,.
\end{aligned}
$$
\end{pro}

Note that the inequality in the critical case is stronger, since the underlying
domain is enlarged. 
By Lemma \ref{lem:Kcompare}, it is sufficient to prove Proposition 
\ref{pro:maximal_functions_inequality}
for $n=0$ (i.e., for $\Lap - \lambda \, I$) and real $\lambda \geqslant -\frac14$, showing that
\begin{equation}\label{eq:zero}
\begin{aligned}
\Mfrak_a^{(0,\lambda)}\,g &\leqslant C(a, \lambda) \, \Mcal g\,\quad 
\text{if }\;\lambda > -\frac14\,,\AND\\
\Mfrak_{(a)}^{(0,\lambda)}\,g &\leqslant C(a, \lambda) \, \Mcal g\,\quad 
\text{if }\;\lambda = -\frac14\,.
\end{aligned}
\end{equation}
That is, we only need to work with the standard spherical functions 
$\Phi(z \,|\, \lambda)$.

In this case, the proof is practically folklore when $\lambda > -\frac14$, 
see \cite{Zygmund}, \cite{Kor} or \cite{Mich},  while the 
extended version for $\lambda = -\frac14$ is contained in a note by {\sc Sj\"ogren}~\cite{Sj}.
However, the extent to which the proof is folklore is such that it is hard to find
a simple version for the hyperbolic Laplacian, and the note \cite{Sj} is also not 
very easily accessible. Therefore, for the reader's convenience, we include
a simple proof (more direct via partial integration than typical versions which decompose 
the unit circle in countably many arcs).

\begin{proof}[{\bf Proof of Proposition \ref{pro:maximal_functions_inequality}}]$\,$
\\
{\bf 1)} As mentioned, it is sufficient to work with $n=0$ and $\lambda$ real. 
Then $\mathcal{K}_{n,\lambda}(z,\xi)\, d\xi$ is a probability measure, and we 
have rotation invariance:
$\;\mathcal{K}_{n,\lambda}(e^{\im  \alpha} z,e^{\im  \alpha}\xi) 
= \mathcal{K}_{n,\lambda}(z,\xi)\,$. Thus, it is 
enough to prove the inequality at $\zeta=1$. 
\\[5pt]
{\bf 2)} For the rest of this proof, we set
\begin{equation}\label{eq:t} 
t = s(\lambda) + 1/2 \AND R^* = \begin{cases} 1\,,&\text{if }\; \lambda > -\frac14\,,\\[.2cm]
                                                             R\,,&\text{if }\; \lambda = -\frac14\,.
                                               \end{cases}
\end{equation}
Thus $t > 1/2$ when $\lambda > -\frac14$, and $t = 1/2$ when
$\lambda = -\frac14$. We can subsume both admissible regions $\Gamma_a(1)$ and $\Gamma_{(a)}(1)$
under
$$
\bigl\{ z = r\,e^{\im \alpha} \in \DD : 
\rho\bigl(z,[0,1] \bigr) \leqslant a + \log R^*\bigr\}\,,
$$
where, as always, $R = \log \dfrac{1+r}{1-r}\,$.
Elementary computations with hyperbolic distance yield for $z = r \,e^{\im \alpha}$
with $r < 1$ and $|\alpha| \leqslant \pi$ that for any $a \geqslant 0$,
\begin{equation}\label{eq:rho}
\rho\bigl(z,[0,1] \bigr) \leqslant a + \log R^*\iff
\underbrace{\frac{1-r^2}{2r} \,\sinh (a+\log R^*)}_{\displaystyle =: K_r}\geqslant 
\begin{cases} |\sin \alpha|\,,&\text{if }\; |\alpha| < \pi/2\,\\
              1\,,&\text{if }\; |\alpha| \geqslant \pi/2\,.  
\end{cases}           
\end{equation}
If $|\alpha| \geqslant \pi/2\,$, this is the same as 
$R \leqslant a + \log R^*\,$.
Thus, there is $r_a \in (0\,,\,1)$  such that when $r \geqslant r_a$ 
then the second of the above cases is excluded, and that $\sinh (a+ R_a) >1$, 
where $R_a = \rho(r_a,0)$.
\\[5pt]
{\bf 3)}  
In view of \eqref{eq:taur}, we need to show that for 
$z = r\,e^{\im \alpha}$ 
in the range of \eqref{eq:rho},
\begin{equation}\label{eq:integral}
\frac{e^R}{2\pi\, R^*} \int_{-\pi}^{\pi} |g(e^{\im \phi})|\, 
\frac{1}{\bigl(1+ \frac{4r}{(1-r)^2}\sin^2\frac{\phi-\alpha}{2}\bigr)^t}\,  d\phi 
\leqslant C(a,\lambda) \, \mathcal{M}g(1).
\end{equation}
If $r \in [0\,,\,r_a)$ then the kernel of the above integral is bounded,
so that the estimate is immediate with a suitable value of $C(a,\lambda)$. So we now
consider the case $r \geqslant r_a\,$, where we know that $|\alpha| < \pi/2$, and $\alpha \to 0$
as $r \to 1$. For $\phi \in (-\pi\,,\,\pi]$, we have $|\phi-\alpha|/2 \leqslant 3\pi/4$, whence there
is $\kappa \in  (0\,,\,1)$ such that 
$|\sin \frac{\phi-\alpha}{2}| \geqslant \kappa \, |\frac{\phi-\alpha}{2}|$, as well as
$|\sin \alpha| \geqslant \kappa \, |\alpha|$. We may assume that $r_a$ is such that 
$2K_r \leqslant \pi$ for $r \geqslant r_a\,$. Now the left hand
side of \eqref{eq:integral} is bounded above by 
$$
\frac{e^R}{2\pi\, R^*} \int_{-\pi}^{\pi} |g(e^{\im \phi})|\,
\frac{1}{\bigl(1+ 
 \frac{r}{\kappa^2(1-r)^2} (\phi-\alpha)^2\bigr)^t}\,  d\phi \,.
$$
We decompose the last integral into 
$\int_{|\phi| < 2K_r/\kappa} + \int_{|\phi| \geqslant 2K_r/\kappa}\,$: for the first part,
$$
\begin{aligned}
\frac{e^R}{2\pi\, R^*} & \int_{-2K_r/\kappa}^{2K_r/\kappa} |g(e^{\im \phi})|\,
\frac{1}{\bigl(1+ 
 \frac{r}{\kappa^2(1-r)^2} (\phi-\alpha)^2\bigr)^t}\,  d\phi \\[.2cm]
 &\leqslant 
 \frac{e^R}{2\pi\, R^*} \int_{-2K_r/\kappa}^{2K_r/\kappa} |g(e^{\im \phi})| \, d\phi 
\leqslant \frac{4K_r\, e^R}{2\pi\,\kappa\, R^*} \, \mathcal{M}g(1) \leqslant \frac{4e^a}{2\pi\,\kappa\, r_a} 
\, \mathcal{M}g(1)\,. 
\end{aligned}
$$
For the second part, we set 
$$
H(\phi) = \int_0^{\phi} \bigl(|g(e^{\im \psi})|+|g(e^{-\im \psi})|\bigr)\, d\psi\,, \quad 
\phi \in [-\pi\,,\,\pi].
$$
Then 
$$
H(\phi) \leqslant 2\phi \, \Mcal g(1)\,.
$$
Observe that, by \eqref{eq:rho} and the choices of $r$ and $\kappa$,
we have $|\alpha| \leqslant K_r/\kappa$, so that  $|\phi| \geqslant 2K_r/\kappa$ implies 
$|\phi - \alpha| \geqslant |\phi|/2$.  We get, with  $c_{\kappa} = (2\kappa)^{2t}/(r_a\,\pi)$, 
and using partial integration
$$
\begin{aligned}
\frac{e^R}{2\pi\, R^*} & \int_{2K_r/\kappa \leqslant |\phi| \leqslant \pi} |g(e^{\im \phi})|
\,\frac{1}{\bigl(1+ \frac{r}{\kappa^2(1-r)^2} (\phi-\alpha)^2\bigr)^t}\,\,  d\phi \\[5pt]
 &\leqslant 
c_{\kappa} \,\frac{(1-r)^{2t-1}}{R^*} \,
\int_{2K_r/\kappa}^{\pi}  \bigl(|g(e^{\im \phi})|+|g(e^{-\im \phi})|\bigr)
\frac{1}{\phi^{2t}}\,  d\phi \\[5pt]
&= c_{\kappa} \,\frac{(1-r)^{2t-1}}{R^*} 
\left( \frac{H(\pi)}{\pi^{2t}} - \frac{H(2K_r/\kappa)}{(2K_r/\kappa)^{2t}} 
  + 2t \int_{2K_r/\kappa}^{\pi} \frac{H(\phi)}{\phi^{2t+1}}\, d\phi\right) \\[5pt]
&\leqslant c_{\kappa} \,\frac{(1-r)^{2t-1}}{R^*} \left(\frac{2\pi}{\pi^{2t}}  
+ 4t \int_{2K_r/\kappa}^{\pi} \frac{1}{\phi^{2t}}\, d\phi\right)
\mathcal{M}g(1)\\[5pt]
&\leqslant \begin{cases} c_{\kappa} 
     \left( \dfrac{2}{\pi^{2t-1}} + \dfrac{4t}{2t-1} \Bigl(\dfrac{\kappa}{\sinh a}\Bigr)^{2t-1}
     \right) \mathcal{M}g(1)\,,&\text{if }\; t > 1/2,\\[.5cm]
     c_{\kappa} \left( \dfrac{2+2\log\pi}{R_a} + 2\right) \mathcal{M}g(1)\,,&\text{if }\; t = 1/2\,,
     \end{cases}
\end{aligned}
$$
because $-\log(2K_r/\kappa) \leqslant R$ by the choices of $r_a$ and $\kappa$.
This concludes the proof.
\end{proof}

As a consequence, we have the following convergence theorem (in the discrete setting of trees, see \cite[Theorem 1, Theorem 3]{KP} for $\lambda$-harmonic functions, and \cite[Theorem 4.6]{Sava&Woess} for regular trees and $\lambda$-polyharmonic functions).
\begin{dfn}\label{def:admconv}
A function $f\colon\DD\to\C$ \emph{converges admissibly, ($\equiv$ non-tangentially)}  
resp. \emph{enlarged-admissibly} 
at a boundary point $\xi\in\partial\DD$ if, for every $a\geqslant 0$,  the limit 
$$
\lim_{\Gamma_a(\xi) \ni z\to\xi} f(z)\,,\quad  \text{resp.} \quad 
\lim_{\Gamma_{(a)}(\xi) \ni z\to\xi} f(z)
$$ 
exists and is finite. 
\end{dfn}

\begin{theorem}[Fatou theorem for $\lambda$-polyharmonic functions]\label{thm:Fatou}
Let $\lambda\in\C\setminus (-\infty\,,\,-\frac14)$ and $\nu$ a Borel measure on $\partial \DD$. 
Then the normalised $\lambda$-polyharmonic function
$$
\frac{\Poiss_{n,\lambda}\nu(z)}{\Phi_n(z \,|\, \lambda)} 
\qquad (\,|z| \geqslant r_{n,\lambda}\,)
$$ 
converges admissibly at $\leb$-almost every point in $\partial \DD$. 
The limit is  the Radon--Nykodim derivative $d\nu^{\text{ac}}/d\leb$, where (recall) $\leb$ is
the normalised arc length measure, and $\nu^{\text{ac}}$ is the absolutely continuous part 
of $\nu$.

If $\lambda = -\frac14$, convergence is even enlarged-admissible.
\end{theorem}

\begin{proof} It is well-known that we can assume $\nu$ to be absolutely continuous 
with respect to $\leb$. 
We prove this fact for the sake of completeness: if $\nu$ is singular with respect to 
$\leb$, then there is an $\leb$-null set $E\subset\partial\DD$ such that $\nu(\DD\setminus E)=0$.
For every $\ep > 0$, let $\{A_j\subset\partial\DD\colon j=1,\dotsc,n\}$ be a finite collection of 
open arcs covering $E$ with $\leb\bigl(\bigcup_{j=1}^n A_j\bigr)<\ep\,$. Let 
$U=\bigcup_{j=1}^n A_j$. Then 
$$
\left|\frac{f(z)}{\Phi_n(z\barra\lambda)}\right| \leqslant 
\int_U \frac{\bigl|P(z,\xi \,|\, \lambda)\,\hor(z,\xi \,|\, \lambda)^n\bigr|}
{|\Phi_n(z\barra\lambda)|}\,d\nu(\xi)\,,
$$
which tends to $0$ when $z \to \zeta \in \partial \DD \setminus U$ by Lemma \ref{lem:tozero}.

So we can assume that $\nu$ is absolutely continuous with respect to the normalised Lebesgue measure 
$\leb$, with  $g=d\nu/d\leb \in L^1(\partial\DD,\leb)$. The rest of the proof is the 
continuous analogue of \cite[Theorems 1 and 3]{KP} and \cite[Theorem 4.6]{Sava&Woess}. 
In brief, we can find a sequence $(g_k)$ in $\mathcal{C}(\partial \DD)$ such that
$\| g - g_k \| < 1/2^k$, and by propositions \ref{prop:HL_is_weak_type_1-1}
and \ref{pro:maximal_functions_inequality},
$$
\sum_k \leb\bigl[\Mfrak_a^{n,\lambda}(g-g_k) \geqslant \ep \bigr] <\infty\,.
$$ 
By the Borel-Cantelli Lemma, 
$$
\lim_{k \to \infty} \Mfrak_a^{n,\lambda}(g-g_k)(\xi) = 0 \quad \text{for $\leb$-almost every }\;
\xi \in \partial \DD.
$$
We can now apply Proposition \ref{pro:basic_convergence}(i) to each of 
the $g_k$ to get the proposed convergence at all those points $\xi\in \partial \DD\,$.
\end{proof}

Let us call a $\lambda$-polyharmonic function $f$ \emph{regular} if all the analytic functionals
$\nu_k$ in the boundary representation of Theorem \ref{thm:poly}, or better (equivalently) formula 
\eqref{eq:poly}, are complex Borel measures on $\partial \DD$. The fact that 
$\Phi_k(r \,|\, \lambda)/\Phi_n(r \,|\, \lambda) \to 0$ for $k < n$, when $r \to 1$, 
yields the following.

\begin{cor}\label{cor:regular}  For $\lambda \in \C\setminus (-\infty\,,\,-\frac14)$, let 
$f$ be a regular $\lambda$-polyharmonic function of order $n+1$ and
$\nu_n$ the highest-order representing measure of $f$ in \eqref{eq:poly}. Then 
$f(z)/\Phi_{n}(z \,|\, \lambda)$ converges admissibly (resp. enlarged-admissibly, if 
$\lambda = -\frac14$)
at almost every $\xi \in \partial \DD$. The limit function 
is the Radon--Nykodim derivative  $d\nu_n^{\text{ac}}/d\leb\,$.
\end{cor}

Compare the following  with \cite{SchBe}.

\begin{cor}\label{cor:radial}  For $\lambda \in \C\setminus (-\infty\,,\,-1/4)$, the only 
radial $\lambda$-polyharmonic functions of order $n+1$ on $\DD$
are the linear combinations of the $\lambda$-polyspherical functions $\Phi_0,\dotsc,\Phi_n$.
\end{cor}

\section{Examples, complements, open problems}\label{sec:exa}

In this section, we study examples of (hyperbolically) harmonic and bi-harmonic
functions which are not regular. We focus on the case $\lambda = 0$,
where of course $\mus(0)+1/2 =1$.

First, for $0 \leqslant r < 1$ and $\xi \in \partial \DD$,
$$
P(r,\xi) = \frac{1-r^2}{1+r^2-r\xi - r/\xi} = \sum_{n \in \Z} r^{|n|}\,\xi^n.
$$
Thus, for any fixed $z =r \, e^{\im \alpha} \in \DD$, the Fourier expansion of the Poisson kernel
in the boundary variable $\xi = e^{\im \phi}$ is
$$
P(r \, e^{\im \alpha}, e^{\im \phi}) = 
\sum_{n \in \Z} r^{|n|}\, e^{-\im n \alpha}\, e^{\im n \phi}.
$$
Next, we want do determine the Fourier expansion of $-\hor(z,\xi)= \log P(z,\xi)$.
By \eqref{eq:extendedP},
$$
\log P(r,\xi) 
= \log(1-r^2) - \log(\xi-r) - \log\Bigl(\frac{1}{\xi}-r\Bigr)
= \log(1-r^2) + \sum_{0 \ne n \in \Z} \frac{r^{|n|}}{|n|}\, \xi^n\,,
$$
so that  we have the Fourier expansion
$$
\log P(r \, e^{\im \alpha}, e^{\im \phi})  = \log(1-r^2) 
+ \sum_{0 \ne n \in \Z} \frac{r^{|n|}}{|n|}\, e^{-\im n \alpha}\, e^{\im n \phi}.
$$
Now we can write 
$$
P(r,\xi) \log P(r,\xi) = \sum_{n \in \Z} r^{|n|}\, d_{|n|}(r)\, \xi^n\,,
$$
since the coefficients of $\xi^n$ and $\xi^{-n}$ must coincide.
For $n \geqslant 0$,
$$
\begin{aligned}
r^n \,d_n(r) &=  r^n \log(1-r^2) + \sum_{0 \ne k \in \Z} r^{|n-k|} \frac{r^{|k|}}{|k|}\\
&= r^n \log(1-r^2) + \underbrace{\sum_{k=1}^{\infty} \frac{r^{n+2k}}{k}}_{\displaystyle = -r^n \log(1-r^2)}
                   + \sum_{k=1}^n \frac{r^n}{k} + \sum_{k=n+1}^{\infty} \frac{r^{2k-n}}{k}\,,\\
\end{aligned}
$$
whence 
\begin{equation}\label{eq:dn}
d_n(r) = \sum_{k=1}^n \frac{1}{k} + \sum_{k=1}^{\infty} \frac{r^{2k}}{k+n}\,.
\end{equation}
Note that $d_n(r)$ is in fact a function of $r^2$.
We have
$$
d_n(r) - d_{n-1}(r) = \sum_{k=0}^{\infty} \frac{(1-r^2)r^{2k}}{k+n} \geqslant 0\,,
$$
so that 
$$
d_n(r) = d_0(r) + (1-r^2)\sum_{k=0}^{\infty} \frac{r^{2k}}{k} \sum_{m=1}^n\frac{k}{k+n}\,,
$$
whence
\begin{equation}\label{eq:dnlog}
\log \frac{1}{1-r^2} = d_0(r) \leqslant d_n(r) \leqslant 
             \displaystyle \bigl( 1 + n(1-r^2) \bigr)\,\log \frac{1}{1-r^2}\,.
\end{equation}
We conclude that the Fourier expansion of the biharmonic Poisson kernel is 
$$
P_1(r \, e^{\im \alpha}, e^{\im \phi}) 
=  P(r \, e^{\im \alpha}, e^{\im \phi})\, \log P(r \, e^{\im \alpha}, e^{\im \phi}) 
=  \sum_{n \in \Z} r^{|n|}\,d_{|n|}(r)\, e^{-\im n \alpha}\, e^{\im n \phi}.
$$
\begin{dfn}\label{def:assoc}
 Let $h(z)$ be a harmonic function on $\DD$ and let $\nu^h$ be 
 the analytic functional on $\partial \DD$ in its Poisson representation.
 The \emph{associated biharmonic function} is
 $$
 f_h(z) = \int_{\partial \DD} P_1(z,\xi)\, d\nu^h(\xi)\,.
 $$
\end{dfn}

Now, let us start with an analytic, whence harmonic function 
\begin{equation}\label{eq:hz}
h(z) = \sum_{n=0}^{\infty} h_n\, z^n\,, \quad  
\limsup_{n \to \infty} |h_n|^{1/n} \leqslant 1
\end{equation}
on $\DD$. We compute the Fourier coefficients
$\nu_n^h = \nu^h(e^{-\im  n\phi})$ of the corresponding 
analytic functional $\nu^h$, that is  
$$
h(z) = \int_{\partial\DD} P(z,\xi)\, d\nu^h(\xi) = 
\sum_{n \in \Z} r^{|n|}\, e^{-\im n \alpha}\, \overline{\nu_n^h}\,,
$$
where $z = r \, e^{-\im \alpha}$. Comparison with 
$$
h(z) = \sum_{n=0}^{\infty} h_n\, r^n\,  e^{\im  n\alpha}
$$
yields
\begin{equation}\label{eq:nunh}
\nu_n^h = \begin{cases} \overline{h\,}_{\!-n}\,,&\text{if }\; n \leqslant 0\,,\\
                      0\,,&\text{if }\; n > 0\,. 
        \end{cases}
\end{equation}
Then the associated biharmonic function is
$$
f_h(z) = \sum_{n=0}^{\infty} h_n\, d_n(|z|)\, z^n\,.
$$
Below we shall use the fact that \eqref{eq:dnlog} implies
\begin{equation}\label{eq:nhn}
\left| f_h(z) - \log\frac{1}{1-r^2}h(z) \right| 
\leqslant (1-r^2) \sum_{n=1}^{\infty} n \,|h_n|\, r^n
\end{equation}
with $\log\frac{1}{1-r^2} \sim R$ as $r = |z| \to 1$. 
Moreover, the real part
$$
\Re h(z) = \sum_{n=0}^{\infty} h_n\, r^n \,
\frac{h_n e^{\im n \alpha} + \overline{h\,}_{\!n}e^{-\im n \alpha}}{2}
$$
is harmonic, and the Fourier coefficients of the corresponding 
analytic functional $\nu^{\Re h}$ are 
$$
\nu_n^{\Re h} = \begin{cases} \Re h_0\,,&\text{if }\; n=0,\\[3pt]
                              h_n/2\,,&\text{if }\; n>0,\\[3pt]
                              \overline{h\,}_{\!-n}/2\,,&\text{if }\; n<0\,.
                \end{cases}
$$ 
(Indeed, every real harmonic function arises as the real part of an analytic function.)
It follows that
$$
f_{\Re h(z)}  = -\int_{\partial \DD} P_1(z,\xi)\, d\nu^{\Re h}(\xi)
= \Re f_h(z)\,.
$$

\smallskip

\begin{disc}\label{disc:bounded}
Euclidean and hyperbolic harmonic functions coincide. We know that if a 
harmonic function $h$ is bounded then its representing analytic functional
is in fact a measure with bounded density with respect to the Lebesgue measure
on $\partial \DD$. Euclidean biharmonic functions are of the form $h_1(z) + r^2 h_2(z)$
with $h_1\,,h_2$ harmonic. The Euclidean biharmonic kernel is  $r^2 P(z,\xi)$, and no
normalisation is in place. In this context, {\sc Mazalov}~\cite{Maz} provides an
example of a \emph{bounded} biharmonic function which has no radial limits at the boundary.
The function is of the form $(1-r^2)\,\Re h(z)$ where $h(z)$ is as in \eqref{eq:hz} with
a lacunary sequence of coefficients $h_n \geqslant 0$. It is bounded, while $h(z)$ is not bounded.
What would be an analogue in the hyperbolic case$\,$? Since the biharmonic kernel needs to
be normalised by $\Phi_1(z|0) \sim R \sim |\log(1-r)|$ as $R \to \infty$ (recall that 
$R= \rho(z,0)$ and $r = |z|$), we would look for a
hyperbolically biharmonic function $f(z)$ such that $f(z)/|\log(1-r)|$ is bounded but has 
no radial limits.

\end{disc}

Now suppose that $h(z)$ is such that $h_n \geqslant 0$ and such that $f(z)/|\log(1-r)|$ is bounded.
Then we have by \eqref{eq:dnlog} for $|z| = r$ 
$$
|h(z)| \leqslant h(r) \leqslant 
\sum_{n=0}^{\infty} h_n \frac{d_n(r)}{d_0(r)} r^n 
\leqslant \frac{f_h(r)}{|\log(1-r)|}\,. 
$$
Hence also the harmonic function $h(z)$ is bounded, the (common) representing analytic 
functional of $h(z)$ and $f_h(z)$ is a measure with bounded density and 
$f_h(z)/|\log(1-r)|$ has admissible limits almost everywhere on $\partial \DD$ by 
Theorem \ref{thm:Fatou}.  It is also worth mentioning 
that $f_h(z)/|\log(1-r)| - h(z)$ has admissible 
limit $0$ almost everywhere.\hss\qed

\begin{ques}\label{question:measure}
Let $\nu$ be an analytic functional on $\partial \DD$ and 
$$
f(z) = \Poiss_{1,0}\nu(z) = \int_{\DD} P_1(z,\xi) \, d\nu(\xi)
$$
be such that $f(z)/R$ is bounded. Is it true that $\nu$ must be a measure with bounded
density$\,$?
\end{ques}

\medskip

\noindent
\textbf{On an example of Borichev.\footnote{We acknowledge literature hints of 
Fausto Di Biase (Pescara) which led us to the work of Borichev. We are particularly 
grateful to Alexander Borichev (Marseille) who indicated this clever example to us.}}\\
The following interesting example is closely related to the results and methods  
in the paper by {\sc Borichev et al.}~\cite{Bor}.

\begin{pro}\label{pro:borichev} There is a harmonic function $h(z)$ on $\DD$ such that
$h(z)/|\log(1-r)|$ is bounded, but has no radial limits at any point of $\partial \DD$. 
\end{pro}
 
Of course, this function is also $\Lap$-biharmonic, so it is an example of a biharmonic
function in the sense of Discussion \ref{disc:bounded}. 

\smallskip

We now provide the details of the proof of Proposition \ref{pro:borichev}. 

\smallskip

The set $K = \bigl\{ \frac{t+1}{6}\, e^{3\pi t\im} : t \in [0\,,\,1] \bigr\}$ is a spiral
winding one and a half times around the origin with radius $r$ varying from $1/6$ to $1/3$. 
Using Runge's approximation theorem as for example stated by 
{\sc Rudin}~\cite[Thm. 13.7]{Rudin}, one finds a homogeneous polynomial
\begin{equation}\label{eq:hRunge}
p(z) = \sum_{j=1}^s p_j\,z^j
\quad \text{with}\quad |p(z) - 5/3| < 2/3 \; \text{ for all } \;z \in K\,.    
\end{equation}
It is of course harmonic, and on $\DD$,
$$
|p(z)| \leqslant B\,|z| \,,\quad \text{where}\quad B = \sum_{j=1}^s |p_j|\,.
$$
We construct the function
\begin{equation}\label{eq:hBorichev}
h(z) = \sum_{k=1}^{\infty} k! \, p\bigl(z^{2^{k!}}\bigr)\,, \quad z \in \DD\,.
\end{equation}
Note that $p(0)=0$. 
\begin{lem}\label{lem:logbound}
The series defining $h(z)$ converges absolutely in $\DD$, and
$$
\sup_{z \in \DD} \frac{|h(z)|}{\bigl|\log(1-|z|)\bigr|} < \infty\,.
$$
\end{lem}

\begin{proof} 
Consider
$$
f(z) = \sum_{n=1}^{\infty} \frac{z^n}{n} = -\log(1-z). 
$$
For real $r \in [0\,,\,1)$, the sequence $(r^n/n)$ is decreasing, whence
$$
\begin{aligned}
f(r) &= \sum_{l=0}^{\infty} \sum_{n=2^l}^{2^{l+1}-1} \frac{r^n}{n} \geqslant 
\sum_{l=0}^{\infty} 2^l \frac{r^{2^{l+1}-1}}{2^{l+1}-1} \geqslant 
\frac{1}{2} \sum_{l=1}^{\infty} r^{2^l}
= \frac12\sum_{k=1}^{\infty} \sum_{l=k!}^{(k+1)!-1} r^{2^l} \\
&\geqslant \frac12\sum_{k=1}^{\infty} \bigl((k+1)!-k!\bigr) r^{2^{(k+1)!-1}} 
\geqslant \frac12\sum_{k=1}^{\infty} \bigl(1 - \frac{1}{k+1}\bigr) (k+1)!\, r^{2^{(k+1)!}} 
\geqslant \frac14 \sum_{k=2}^{\infty} k!\, r^{2^{k!}}\,.
\end{aligned}
$$
Therefore
$$
\sum_{k=1}^{\infty} k!\, |z|^{2^{k!}} \leqslant |z|^2 + 4 \bigl|\log(1-|z|)\bigr| 
\leqslant c \, \bigl|\log(1-|z|)\bigr|\,.
$$
Now
$$
|h(z)| 
\leqslant \sum_{j=1}^{s} |p_j| \sum_{k=1}^{\infty} k!\, |z|^{j\,2^{k!}} 
\leqslant c \sum_{j=1}^{s} |h_j| \bigl|\log(1-|z|^j)\bigr|
\leqslant c\, B \,\bigl|\log(1-|z|)\bigr|,
$$
as stated.
\end{proof}

\begin{lem}\label{lem:limsup}
For every $\alpha \in (-\pi\,,\pi]$,
$$
\limsup_{r \to 1} \frac{|h(r \, e^{\im \alpha})|}{\bigl|\log(1-r)\bigr|} \geqslant 1\,.
$$
\end{lem}

\begin{proof} We fix $\alpha$ and $N \in \N$ and can find $r_N = r_N(\alpha)$
such that 
$$
w_N = r_N\, e^{\im \alpha\, 2^{N!}} \in K\,.
$$
Thus,
$$
1/6 \leq r_N\leqslant 1/3\,, \quad \bigl|p(w_N) - 5/3| < 2/3\,,\AND \Re\bigl(p(w_N)\bigr) > 1. 
$$
Now we choose
$$
z_N = r_N^{1/2^{N!}}\,e^{\im \alpha}\,,  
$$
so that $|z_N| \to 1$, and we consider 
$$
h(z_N) - N!\,p\bigl(z_N^{2^{N!}}\bigr) 
= \sum_{k=1}^{N-1} k!\, p\bigl(z_N^{2^{k!}}\bigr) 
+ \sum_{k=N+1}^{\infty} k!\, p\bigl(z_N^{2^{k!}}\bigr). 
$$
We can estimate the first sum by 
$$
\left| \sum_{k=1}^{N-1} k!\, p\bigl(z_N^{2^{k!}}\bigr)  \right| \leqslant B \sum_{k=1}^{N-1} k! 
\leqslant 2B\,(N-1)! = \frac{2B}{N} N!\,.
$$
and, by setting $x_N = (1/3)^{N!}$ and using that $2^n \geqslant n$, the second sum is majorized as follows:
$$
\begin{aligned}
\left| \sum_{k=N+1}^{\infty} k!\, p\bigl(z_N^{2^{k!}}\bigr) \right|
&\leqslant B\, N! \sum_{k=N+1}^{\infty} \frac{k!}{N!}\, x_N^{2^{(k!/N!)-1}} 
\leqslant B\, N! \sum_{k=N+1}^{\infty} \frac{k!}{N!}\, x_N^{(k!/N!)-1}\\[.25cm]
&\leqslant B\, N! \sum_{m=N+1}^{\infty} m\, x_N^{m-1} 
= B\, N! \, \underbrace{\frac{(N+1)x_N^N - Nx_N^{N+1}}{(1-x_N)^2}}_{\displaystyle 
                                   \to 0\,,\; \text{ as } \; N \to \infty}\,. 
\end{aligned}
$$
It follows that
$$
\Bigl|h(z_N) - N!\,\underbrace{p\bigl(z_N^{2^{N!}}}_{\displaystyle p(w_n)}\bigr)\Bigr| 
\leqslant N!\, \epsilon_N\,, 
\quad \text{where }\; \epsilon_N \to 0 \; \text{ as } \; N \to \infty\,.
$$
Now note that $1 < \log(1/r_N) < 2$. Using this and the fact that for $x >0$
one has $x-x^2/2 \leqslant 1-e^{-x} \leqslant x$, one obtains that 
$$
1-|z_N| = 1- e^{-2^{-N!} \log(1/r_N)} 
\begin{cases} < 2\cdot 2^{-N!}\\ > \frac12 \cdot 2^{-N!}.
\end{cases}
$$
We deduce that 
$\bigl|\log(1-|z_N|)\bigr| \sim N!\,$, whence 
$$
\left| \frac{h(z_N)}{\bigl|\log(1-|z_N|)\bigr|} - p(w_N) \right| \to 0  
\quad \text{ as } \; N \to \infty\,.
$$
Since $|p(w_N)| \geqslant \Re\bigl(p(w_N)) > 1$, the statement is proved.
\end{proof}

\medskip

\begin{lem}\label{lem:zero}
$\qquad \displaystyle
\lim_{N \to \infty} \sup_{|z| = 1-2^{-N!\sqrt{N}}} \frac{h(z)}{\bigl|\log(1-|z|)\bigr|} = 0\,.
$
\end{lem}

\begin{proof}
It is sufficient to show that
$$
\lim_{N \to \infty} \frac{h_0\bigl(1-2^{-N!\sqrt{N}}\bigr)}{N!\sqrt{N}} = 0\,,\quad 
\text{where} \quad
h_0(z) = \sum_{k=1}^{\infty} k!\, z^{2^{k!}}.
$$
Similarly as above, using the fact that  $(1-2^{-N!\sqrt{N}})^{2^{N!\sqrt{N}}} \leqslant 1/e$ 
and setting $y_N = e^{-N!}$, we find
$$
\begin{aligned}
h_0\bigl(1-2^{-N!\sqrt{N}}\bigr) &\leqslant 
\sum_{k=1}^{N} k!\,  +  \sum_{k=N+1}^{\infty} k!\,\bigl(\frac{1}{e}\bigr)^{2^{k!-N!\sqrt{N}}}\\ 
&\leqslant 2\,N! 
+ N! \sum_{k=N+1}^{\infty} \frac{k!}{N!}\,(y_N)^{(k!/N!)-1} y_N^{1-\sqrt{N}}\\ 
&\leqslant 2\,N! 
+ N! \sum_{m=N+1}^{\infty} m\,(y_N)^{m-1} y_N^{1-\sqrt{N}}\\ 
&= 2\,N! 
+ N! \underbrace{\frac{(N+1)y_N^{N+1-\sqrt{N}} - Ny_N^{N+2-\sqrt{N}}}{(1-y_N)^2}}_{\displaystyle 
                                   \to 0\,,\; \text{ as } \; N \to \infty}. 
\end{aligned}                                   
$$
Divided by $N!\sqrt{N}$, this tends to $0$.
\end{proof}

Lemmas \ref{lem:logbound}, \ref{lem:limsup} and \ref{lem:zero} prove Proposition 
\ref{pro:borichev}. With very small adaptations, one verifies that the same statement
holds for the harmonic function $\Re h(z)$.
We conclude with the following observation, which should not look too surprising. (Recall that 
for $\lambda = 0$, $\Phi_2(z|0) \sim R^2/2$, as $r = |z| \to 1$.)

\begin{lem}\label{lem:fh-borichev}
The biharmonic function $f_h$ associated with the function $h$ of  \eqref{eq:hBorichev}
is such that  $f_h(z)/R^2$ is bounded and has no radial limits at any point of $\partial \DD$.
\end{lem}

\begin{proof}
Similarly to Lemma \ref{lem:logbound}, for $0 \leqslant r < 1$,
$$
\begin{aligned} \frac{1}{1-r} \log \frac{1}{1-r} 
&= \sum_{n=1}^{\infty} \left(\sum_{j=1}^n \frac{1}{j}\right) r^n
\geqslant \sum_{k=1}^{\infty} 
\left(\sum_{n=2^{k!}}^{2^{(k+1)!}-1} \sum_{j=1}^n \frac{1}{j}\right) r^{2^{(k+1)!}} \\
&\stackrel{(*)}{\geqslant} \tilde c \sum_{k=1}^{\infty} (k+1)! \, 2^{(k+1)!} \, r^{2^{(k+1)!}}. 
\end{aligned}
$$
For $(*)$, see below.
We get via \eqref{eq:nhn}
$$
\left| f_h(z) - \log\frac{1}{1-r^2}\,h(z) \right| \leqslant (1-r^2) \sum_{j=1}^N j\,|p_j| 
\sum_{k=1}^{\infty} k! \, 2^{k!}\, r^{j2^{k!}} \leqslant \widetilde C \log \frac{1}{1-r}
$$
for a suitable constant $\widetilde C > 0$. The statement follows by dividing by 
$\bigl(\log \frac{1}{1-r}\bigr)^2$ and applying Proposition \ref{pro:borichev}.

Let us prove $(*)$:
$$
\begin{aligned}
Q &:=\frac{1}{(k+1)!\, 2^{(k+1)!}}\sum_{n=2^{k!}}^{2^{(k+1)!}-1} \sum_{j=1}^n \frac{1}{j}  
\\
&\phantom{:} = 
\frac{1}{(k+1)!\, 2^{(k+1)!}}\, 2^{k!} \sum_{j=1}^{2^{k!}} \frac{1}{j} \; + \; 
\frac{1}{(k+1)!\, 2^{(k+1)!}}\sum_{j=2^{k!}+1}^{2^{(k+1)!}-1} \frac{2^{(k+1)!}-j}{j}.
\end{aligned}
$$
As $k \to \infty$, the first term behaves like 
$$
\frac{ 2^{k!}\, k! \,\log 2}{(k+1)!\, 2^{(k+1)!}} \to 0\,,
$$
while the second term is
$$
\frac{1}{(k+1)!} \sum_{j=2^{k!}+1}^{2^{(k+1)!}-1} \Bigl(\frac{1}{j} - \frac{1}{2^{(k+1)!}}\Bigr) 
\to \log 2\,.
$$
Thus, the quotient $Q$ is bounded below by some $\tilde c > 0$ for all $k$.
\end{proof}

\newpage

\end{document}